\documentclass[a4paper,11pt,reqno]{amsart}

\usepackage{amsfonts}
\usepackage{amssymb}
\usepackage{amscd}
\usepackage{amsthm}
\usepackage{a4wide}
\usepackage{mathrsfs}
\usepackage[applemac]{inputenc}
\usepackage{amsmath}
\usepackage{amssymb}
\usepackage{amsthm}
\usepackage{textcomp}
\usepackage{graphicx}
\usepackage{enumerate}
\usepackage{mathrsfs}
\usepackage{frcursive}
\usepackage{tikz}
\usepackage[cyr]{aeguill}
\usepackage{xspace}
\usepackage{hyperref}
\usepackage{appendix}
\usepackage{esint}
\usepackage{cancel}
\usepackage{calc}

\newtheorem{defn}{Definition}[section]
\newtheorem{lemma}[defn]{Lemma}
\newtheorem{prop}[defn]{Proposition}
\newtheorem{theo}[defn]{Theorem}
\newtheorem{coro}[defn]{Corollary}

\newtheorem{rk}[defn]{Remark}

\def\Ric{\mathop{\rm Ric}\nolimits}

\def\Rm{\mathop{\rm Rm}\nolimits}
\def\tr{\mathop{\rm tr}\nolimits}

\def\vol{\mathop{\rm vol}\nolimits}
\def\eucl{\mathop{\rm eucl}\nolimits}

\def\vol{\mathop{\rm Vol}\nolimits}

\def\div{\mathop{\rm div}\nolimits}

\def\AVR{\mathop{\rm AVR}\nolimits}

\def\Li{\mathop{\rm \mathscr{L}}\nolimits}

\def\Ric{\mathop{\rm Ric}\nolimits}

\def\Rm{\mathop{\rm Rm}\nolimits}
\def\tr{\mathop{\rm tr}\nolimits}

\def\vol{\mathop{\rm vol}\nolimits}
\def\eucl{\mathop{\rm eucl}\nolimits}

\def\vol{\mathop{\rm Vol}\nolimits}

\def\div{\mathop{\rm div}\nolimits}

\def\AVR{\mathop{\rm AVR}\nolimits}

\def\Li{\mathop{\rm \mathscr{L}}\nolimits}
\def\En{\mathop{\rm \mathscr{E}}\nolimits}

\def\supp{\mathop{\rm supp}\nolimits}

\def\R{\mathop{\rm R}\nolimits}

\def\Hom{\mathop{\rm Hom}\nolimits}

\newsavebox\CBox
\newcommand\hcancel[2][0.5pt]{%
  \ifmmode\sbox\CBox{$#2$}\else\sbox\CBox{#2}\fi%
  \makebox[0pt][l]{\usebox\CBox}%
  \rule[0.5\ht\CBox-#1/2]{\wd\CBox}{#1}}

\newcommand{\nor}{\hcancel{\|}}

\title{Weak stability of Ricci expanders with positive curvature operator}

\author[Alix Deruelle and  Tobias Lamm]{Alix Deruelle and Tobias Lamm}
\address[Alix Deruelle]{D\'epartement de math\'ematiques, B\^atiment 425, Facult\'e des sciences, Universit\'e Paris-Sud 11, F-91405, Orsay }
\email{alix.deruelle@math.u-psud.fr}
\address[Tobias Lamm]{Institute for Analysis, Karlsruhe Institute of Technology (KIT), Englerstr. 2, 76131
Karlsruhe, Germany}
\email{tobias.lamm@kit.edu}

\begin{document}
\begin{abstract}
We prove the $L^{\infty}$ stability of expanding gradient Ricci solitons with positive curvature operator and quadratic curvature decay at infinity. 
\end{abstract}

\maketitle

\section{Introduction}

We investigate the stability of expanding gradient Ricci solitons with positive curvature operator along the Ricci flow in the spirit of Koch and Lamm \cite{Koc-Lam-Rou}. The main aim of this approach is twofold : on one hand, it allows to start the Ricci flow from singular initial metrics that are close to a self-similar solution of the Ricci flow in the $L^{\infty}$ sense, on the other hand, it proves the existence of global solutions of the Ricci flow, i.e. solutions that are defined for all positive time. 

We recall that an expanding gradient Ricci soliton is a self-similar solution of the Ricci flow, i.e. it evolves by dilating homotheties and diffeomorphisms generated by the gradient of a smooth function. Formally speaking, it consists of a triplet $(M^n,g,\nabla^g f)$ where $(M^n,g)$ is a complete Riemannian manifold endowed with a smooth function $f:M^n\rightarrow\mathbb{R}$ called the potential function such that the associated Bakry-\'Emery tensor is constantly negative as follows : $$\Ric(g)+\nabla^{g,2}(-f)=-\frac{g}{2}.$$ Alternatively, the corresponding Ricci flow is given, at least formally, by $g(\tau)=(1+\tau)\phi^*_{\tau}g$, for $\tau\in(-1,+\infty)$ where $(\phi_{\tau})_{\tau>-1}$ is a one parameter family of diffeomorphisms generated by the vector field $-\nabla^gf/(1+\tau)$. Expanding gradient (Ricci) solitons with non negative curvature operator are non collapsed, that is, the volume of geodesic balls is Euclidean for any scale and are diffeomorphic to the Euclidean space $\mathbb{R}^n$, where $n$ is the dimension of the underlying manifold. Thanks to the work of Schulze and Simon \cite{Sch-Sim}, they appear as natural blow-down of Ricci flows starting from Riemannian manifolds with bounded non negative curvature operator and positive asymptotic volume ratio. Here are some examples of such geometric structures with non negative curvature operator : 
\begin{enumerate}
\item The Gaussian soliton $\left(\mathbb{R}^n,\eucl,\nabla^{\eucl}\left(\frac{\arrowvert\cdot\arrowvert^2}{4}\right)\right)$, where $\phi_{\tau}(x):=x/(\sqrt{1+\tau})$, $\tau>-1$.\\
\item The Gutperle-Headrick-Minwalla-Schomerus examples \cite{Gut-Hea-Min-Sch} in dimension $2$ : it consists of a one parameter family $(\mathbb{R}^2,g_c,\nabla^{g_c}f_c)_{c\in(0,1)}$ asymptotic to $(C(\mathbb{S}^1),dr^2+(cr)^2d\theta^2,r\partial_r/2)_{c\in(0,1)},$ with rotational symmetry.\\
\item The Bryant examples [Chap. $1$,\cite{Cho-Lu-Ni-I}] correspond to the previous examples in higher dimensions : again, it consists of a one parameter family $(\mathbb{R}^n,g_c,\nabla^{g_c}f_c)_{c\in(0,1)}$ asymptotic to $(C(\mathbb{S}^{n-1}),dr^2+(cr)^2g_{\mathbb{S}^{n-1}},r\partial_r/2)_{c\in(0,1)}$ with rotational symmetry.\\
\item In \cite{Der-Asy-Com-Egs} and \cite{Der-Smo-Pos-Met-Con}, the first author proved the existence and uniqueness of deformations of gradient expanders with positive curvature operator smoothly coming out of a metric cone over a smooth Riemannian sphere $(X,g_X)$ satisfying $\Rm(g_X)\geq  1$.
\end{enumerate}

Finally, we mention that there is a one-to-one correspondence between the following geometric notions : for such expanders, the asymptotic cone (in the Gromov-Hausdorff sense) is the singular initial condition of the corresponding Ricci flow : see the introduction of \cite{Der-Smo-Pos-Met-Con} for an explanation.

Now, the Ricci flow is well-known to be a degenerate quasilinear parabolic equation : this comes from its invariance under the whole diffeomorphism group of the underlying manifold. That is why it is more convenient to study the so called (time dependent) DeTurck Ricci flow : let $(M^n,g_0(t))_{t\geq 0}$ be an expanding gradient Ricci soliton, we consider the following system,

\[
\left\{
\begin{array}{rl}
&\partial_tg=-2\Ric(g(t))+\Li_{V(g(t),g_0(t))}(g(t)), \quad \mbox{on $M^n\times (0,+\infty)$,}\\
&\\
& g(0):=g_0(0)+h,\quad \mbox{$\|h\|_{L^{\infty}(M^n,g_0)}$ small},
\end{array}
\right.
\]
where $V(g(t),g_0(t))$ is a vector field defined by 
\begin{eqnarray*}
V(g(t),g_0(t)):=\div_{g_0(t)}(g(t)-g_0(t))-\frac{1}{2}\nabla^{g_0(t)}\tr_{g_0(t)}(g(t)-g_0(t)).
\end{eqnarray*}

See section \ref{sec-equ-flo} for the equivalence with other flows that appear naturally in this setting. 

We are now in a position to state the main result of this paper : 
\begin{theo}\label{main-theo}
Let $(M^n,g_0(t))_{t\geq 0}$ be an expanding gradient Ricci soliton with positive curvature operator and quadratic curvature decay at infinity, i.e.
\begin{eqnarray*}
\Rm(g_0)>0,\quad \arrowvert\Rm(g_0)\arrowvert_{g_0}(x)\leq \frac{C}{1+d_{g_0}^2(p,x)},\quad \forall x\in M,
\end{eqnarray*}
for some point $p\in M$ and some positive constant $C$ (depending on $p$). Then there exists a positive $\epsilon$ such that for any metric $g\in L^{\infty}(M,g_0)$ satisfying $\|g-g_0\|_{L^{\infty}(M,g_0)}\leq\epsilon$, there exists a global solution $(M^n,g(t))_{t\geq 0}$ to the DeTurck Ricci flow  with initial condition $g$. Moreover, 
\begin{enumerate}
\item the solution is analytic in space and time and the following holds for any nonnegative indices $\alpha,\beta$ : 
\begin{eqnarray*}
\sup_{x\in M,t>0}\arrowvert(t^{1/2}\nabla^{g_0(t)})^{\alpha}(t\partial_t)^{\beta}(g(t)-g_0(t))\arrowvert_{g_0(t)}\leq c_{\alpha,\beta}\arrowvert g(0)-g_0\arrowvert_{L^{\infty}(M,g_0)},
\end{eqnarray*}
\item this solution is unique in $B_X(g_0,\epsilon)$ where $X$ is a Banach space defined in section \ref{sec-fct-spa}.
\end{enumerate}
\end{theo}

Let us comment the assumptions and the proof of theorem \ref{main-theo}. The condition on the sign of the curvature operator is global whereas the one on the decay of the curvature is asymptotic (or local at infinity) : it ensures in particular that the asymptotic cone is $C^{1,\alpha}$ for any $\alpha\in(0,1)$ as shown in \cite{Che-Der}. Both assumptions are crucial to estimate the homogeneous equation associated to the DeTurck Ricci flow (section \ref{sec-est-hom-equ}) : 

\[
\left\{
\begin{array}{rl}
&\partial_th=L_t h, \quad \mbox{on $M^n\times (0,+\infty)$,}\\
&\\
& h(0)\in S^2T^*M,\quad\|h(0)\|_{L^{\infty}(M^n,g_0)} <+\infty,
\end{array}
\right.
\]
where $(L_t)_{t\geq 0}$ is the one parameter family of time dependent Lichnerowicz operators associated to the background expander $(M^n,g_0(t))_{t\geq0}$ defined on the space $S^2T^*M$ of symmetric $2$-tensors on $M$ by : 
\begin{eqnarray*}
&&L_th:=\Delta_{g_0(t)}h+2\Rm(g_0(t))\ast h-\Ric(g_0(t))\otimes h-h\otimes \Ric(g_0(t)),\\
&& (\Rm(g_0(t))\ast h)_{ij}:=\Rm(g_0(t))_{iklj}h_{kl},\quad h\in S^2T^*M.
\end{eqnarray*}

The main strategy to prove theorem \ref{main-theo} is due to Koch and Lamm \cite{Koc-Lam-Rou} inspired by a previous work on the Navier-Stokes equation by Koch and Tataru \cite{Koc-Tat-Nav-Sto} : it is based on Gaussian bounds for large time of the heat kernel associated to the linearized operator. Of course, such bounds are given for free in the Euclidean case, reinterpreted here as the Gaussian expanding soliton. Therefore, the main challenge of this paper is to get uniform-in-time Gaussian bounds for the heat kernel associated to the time dependent Lichnerowicz operator. We describe succinctly  the main estimates we need to reach such goal :
\begin{enumerate}
\item We establish a $L^{\infty}-L^{\infty}$ bound together with an $L^1-L^1$ bound for this heat operator : only the positivity of the curvature operator is used. See section \ref{sec-hea-ker-est}. This gives more or less an on-diagonal bound for the heat kernel, modulo a Nash-Moser iteration.\\
\item This implies that the heat operator is bounded on $L^p$ for any $p\in[1,+\infty]$ but it does not imply that it is contractive, unless the expander is flat. In particular, the $L^2$ norm of the solution to the homogeneous equation is not decreasing. Despite this lack of monotonicity, we are able to adapt the ideas of Grigor'yan \cite{Gri-Upp-Ene-Dav} to get an off-diagonal bound.
\end{enumerate} 

We mention that the literature on heat kernel bounds in various geometric settings is immense : see the book \cite{Gri-Hea-Ker-Boo} and the references therein in the case of static metrics. On the other hand, more and more refined Gaussian estimates for the heat kernel along the Ricci flow for short time are being investigated intensively under weaker and weaker curvature bounds : see the recent work by Bamler and Zhang \cite{Bam-Zha-Scal-I} and the references therein in the presence of a uniform bound on the scalar curvature only.

Let us describe now briefly where our main theorem \ref{main-theo} is placed in the panorama of stability results along the Ricci flow : the methods are very different whether the manifold is compact (without boundary) or not. Indeed, the $L^2$ spectrum of the Lichnerowicz operator is not discrete anymore in the non compact case.\\

Concerning the compact case : see \cite{Has-Mul-Lam} in the Ricci flat case then \cite{Kro-Sta-Ins} for Ricci solitons.
Their work consists in establishing an adequate Lojasiewicz inequality for the relevant entropies (introduced by Perelman). This approach goes back to the work of Leon Simon. It has been used recently by Colding-Minicozzi for Ricci flat non collapsed non compact Riemmanian manifolds \cite{Col-Min-Ein-Tan-Con}. See also the work of Sesum \cite{Ses-Lin-Dyn-Sta} under integrability assumptions of the moduli space.\\

In the non compact case, the hyperbolic and Euclidean case have first been considered by Schnürer-Schulze-Simon \cite{Sch-Sch-Sim} : their method is based on an exhaustion procedure, solving corresponding Dirichlet DeTurck Ricci flows. The inconvenient of this approach, although being very geometric, is that it does not provide local uniqueness of the solutions in some functional space.
 Bamler \cite{Bam-Sym} proved, among other things, stability in the $L^{\infty}$ sense for non compact symmetric spaces of rank larger than $1$, the analysis is based on \cite{Koc-Lam-Rou}. The first author considered the expanding case in the $L^{\infty}\cap L^2(e^fd\mu(g))$ sense for metrics with non negative curvature operator in the spirit of \cite{Sch-Sch-Sim} : \cite{Der-Sta-Sge}. Again, as the solution is not built by a fixed point argument, we did not know if this solution was unique. By relaxing the assumptions on the initial condition, theorem \ref{main-theo} ensures the well-posedness of the Ricci flow with small initial condition (in the $L^{\infty}$ sense) for such expanders. 

 We end this introduction by explaining the reasons that led us to consider the time-dependent DeTurck Ricci flow.
In \cite{Der-Sta-Sge}, we investigated the asymptotic stability of expanding gradient Ricci solitons with initial condition belonging to a weighted space $L^2(e^fd\mu(g))$ by studying the bottom of the spectrum of the weighted laplacian $\Delta_f$. It is then legitimate to ask only for stability in the $L^{\infty}$ sense. The quasi-linear evolution equation to consider is 
\begin{eqnarray}
\partial_tg=-2\Ric_g-g+\Li_{\nabla^{g_0}f_0}(g)+\Li_{V(g,g_0)}(g),\label{eq-expanding}
\end{eqnarray}
where $(M^n,g_0,\nabla^{g_0}f_0)$ is a fixed background expanding gradient Ricci soliton and $\Li_{V(g,g_0)}(g)$ is the DeTurck's term described previously. 
The linearized operator of (\ref{eq-expanding}) is the weighted (static) Lichnerowicz operator $L:=\Delta_f+2\Rm(g)\ast$.
 In the Euclidean case, the heat kernel associated to $L=\Delta_f$ is the Mehler kernel given by : 
$$k_t(x,y)=\frac{1}{(8\pi\sinh(t/2))^{n/2}}\exp\left\{-\frac{\arrowvert x-y\arrowvert^2}{4(1-e^{-t})}-\frac{1-e^{-t/2}}{1-e^{-t}}\frac{< x,y>}{2}-\frac{nt}{4}\right\}.$$
 The first difficulty lies in estimating the derivatives in time and space of the Mehler kernel associated to a non necessarily Euclidean metric : indeed, the semigroup associated is not strongly continuous nor analytic in $L^{\infty}$ as DaPrato and Lunardi showed \cite{DaP-Lun-Orn-Uhl}. Secondly, the choice of norms is crucial : Koch and Lamm use scale invariants norms together with the implicit fact that the measure is doubling. In this case, this is not true anymore, i.e. $\vol_fB(x,2R):=\int_{B(x,2R)}e^fd\mu(g)$ is not uniformly comparable to $\vol_fB(x,R)$, which would make the analysis harder. That is why the main advantage of the time-dependent DeTurck Ricci flow on (\ref{eq-expanding}) is that the time dependent linearized operator is the unweighted classical Lichnerowicz operator $L_t$.
 Therefore, even if the coefficients are time dependent, the volume measure is now doubling and it is controlled by the time-independent asymptotic volume ratio of the expander : see appendix \ref{sol-equ-sec}. As a consequence, we get a uniform-in-time Mehler bound for the static weighted Lichnerowicz operator $L$ : \ref{subsec-lic-hea-ker}.\\

The organization of the paper is as follows : section \ref{sec-equ-flo} discusses the equivalence of various Ricci flows that are pertinent in this context and describes more precisely each non linear term that appears in the time-dependent DeTurck Ricci flow. Section \ref{sec-fct-spa} defines the relevant Banach spaces. Section \ref{sec-est-hom-equ} is devoted to the analysis of the homogeneous equation associated to the time-dependent DeTurck Ricci flow. Section \ref{sec-hea-ker-est} proves Gaussian bounds for the heat kernel associated to the time-dependent Lichnerowicz operator. Finally, section \ref{sec-est-inhom-equ} proves theorem \ref{main-theo} by a careful study of the inhomogeneous equation associated to this Ricci flow.\\

\textit{The research leading to the results contained in this paper has received  funding from
the European Research Council (E.R.C.) under European Union's Seventh Framework Program
(FP7/2007-2013)/ ERC grant  agreement No. 291060.}

\section{Equivalence of Ricci flows}\label{sec-equ-flo}

An expanding gradient Ricci soliton is a triplet $(M^n,g_0,\nabla^{g_0} f_0)$ where $(M^n,g_0)$ is a Riemannian manifold and $f_0:M\rightarrow\mathbb{R}$ is a smooth function (called the potential function) satisfying
\begin{eqnarray*}
\frac{1}{2}\Li_{\nabla^{g_0} f_0}(g_0)=\Ric(g_0)+\frac{ g_0}{2}.
\end{eqnarray*}
We assume that $(M^n,g_0)$ is complete, this suffices to ensure the completeness of the vector field $\nabla^{g_0} f_0$ : \cite{Zha-Zhu}.

Let us fix an expanding gradient Ricci soliton $(M^n,g_0,\nabla^{g_0}f_0)$. 
In this paper, we study the stability of such gradient Ricci soliton under the Ricci flow, i.e.
\[
\left\{
\begin{array}{rl}
&\partial_tg=-2\Ric(g(t)) \quad\mbox{on}\quad M\times (0,+\infty),\\
&\\
& g(0)=g_0+h,
\end{array}
\right.
\]
where $h$ is a symmetric $2$-tensor on $M$ (denoted by $h\in S^2T^*M$) such that $g(0)$ is a metric.

In this setting, a more relevant flow, called the modified Ricci flow, has been introduced in \cite{Der-Sta-Sge} :
\[
\left\{
\begin{array}{rl}
&\partial_tg=-2\Ric(g(t))-g(t)+\Li_{\nabla^{g_0}f_0}(g(t)) \quad\mbox{on}\quad M\times (0,+\infty),\\
&\\
& g(0)=g_0+h.
\end{array}
\right.
\]
It turns out that these two flows are equivalent, as shown in \cite{Der-Sta-Sge}. Here, we decide to fix the vector field $\nabla^{g_0}f_0$. We could allow the metric to vary by choosing instead $\nabla^{g(t)}f_0$, this would not affect the analysis. Choosing an implicit function $f(t)$ for any positive time should be done by a minimizing procedure involving a relevant entropy in the spirit of Perelman as it is done in \cite{Has-Mul-Lam}. 

In the spirit of the so called DeTurck's trick, as the modified Ricci flow is a degenerate parabolic equation,  we first need to consider the following modified Ricci harmonic map heat flow (MRHF) : 
\[
\left\{
\begin{array}{rl}
&\partial_tg=-2\Ric(g(t))-g(t)+\Li_{\nabla^{g_0}f_0}(g(t))+\Li_{V(g(t),g_0)}(g(t)), \quad\mbox{on}\quad M\times (0,+\infty)\\
&\\
& g(0)=g_0+h
\end{array}
\right.
\]
where $V(g(t),g_0)$ is a vector field defined globally by,
\begin{eqnarray}\label{de-turck-vect}
V(g(t),g_0):=\div_{g_0}g(t)-\frac{1}{2}\nabla^{g_0}\tr_{g_0}g(t).
\end{eqnarray}

Following \cite{Shi-Def} and \cite{Der-Sta-Sge}, the modified Ricci harmonic map heat flow can be written in coordinates as :
\begin{eqnarray*}
\partial_tg_{ij}&=&g^{ab}\nabla^{g_0,2}_{ab}g_{ij}+\nabla^{g_0}_{\nabla^{g_0}f_0}g_{ij}-g^{kl}g_{ip}\Rm(g_0)_{jklp}-g^{kl}g_{jp}\Rm(g_0)_{iklp}\\
&&+g_{ik}\Ric(g_0)_{kj}+g_{jk}\Ric(g_0)_{ki}\\
&&+g^{ab}g^{pq}\left(\frac{1}{2}\nabla^{g_0}_ig_{pa}\nabla^{g_0}_jg_{qb}+\nabla^{g_0}_ag_{jp}\nabla^{g_0}_qg_{ib}\right)\\
&&-g^{ab}g^{pq}\left(\nabla^{g_0}_ag_{jp}\nabla^{g_0}_bg_{iq}-\nabla^{g_0}_jg_{pa}\nabla^{g_0}_bg_{iq}-\nabla^{g_0}_ig_{pa}\nabla^{g_0}_bg_{jq}\right).
\end{eqnarray*}
Define $h(t):=g(t)-g_0$. The previous system can be rewritten as follows
\begin{eqnarray*}
\left(\partial_t-L_0\right)h_{ij}&=&\nabla^{g_0}_a\left(\left(\left(g_0+h\right)^{ab}-g_0^{ab}\right)\nabla^{g_0}_bh_{ij}\right)-\nabla^{g_0}_a\left(g_0+h\right)^{ab}\nabla^{g_0}_bh_{ij}-2\Rm(g_0)_{iklj}h_{kl}\\
&&-g^{kl}g_{ip}\Rm(g_0)_{jklp}-g^{kl}g_{jp}\Rm(g_0)_{iklp}+g_{ik}\Ric(g_0)_{kj}+g_{jk}\Ric(g_0)_{ki}\\
&&+\left(g_0+h\right)^{ab}\left(g_0+h\right)^{pq}\left(\frac{1}{2}\nabla^{g_0}_ih_{pa}\nabla^{g_0}_jh_{qb}+\nabla^{g_0}_ah_{jp}\nabla^{g_0}_qh_{ib}\right)\\
&&-\left(g_0+h\right)^{ab}\left(g_0+h\right)^{pq}\left(\nabla^{g_0}_ah_{jp}\nabla^{g_0}_bh_{iq}-\nabla^{g_0}_jh_{pa}\nabla^{g_0}_bh_{iq}-\nabla^{g_0}_ih_{pa}\nabla^{g_0}_bh_{jq}\right),
\end{eqnarray*}
where $\left(g_0+h\right)^{ab}:=\left((g_0+h)_{ab}\right)^{-1}$ and where $L_0h:=\Delta_{g_0}h+\nabla^{g_0}_{\nabla^{g_0}f_0}h+2\Rm(g_0)\ast h$ is the weighted Lichnerowicz operator acting on symmetric $2$-tensors. Hence, we write the previous system in the following form as in \cite{Koc-Lam-Rou} modulo an extra zeroth order term :
\begin{eqnarray}\label{Static-DeTurck-MRF}
\left(\partial_t-L_0\right)h&=& R_0[h]+\nabla^{g_0}R_1[h],
\end{eqnarray}
where,
\begin{eqnarray*}
R_{0}[h]_{ij}&:=&-2\Rm(g_0)_{iklj}h_{kl}-g^{kl}g_{ip}\Rm(g_0)_{jklp}-g^{kl}g_{jp}\Rm(g_0)_{iklp}\\
&&+g_{ik}\Ric(g_0)_{kj}+g_{jk}\Ric(g_0)_{ki}\\
&&+\frac{1}{2}\left(g_0+h\right)^{ab}\left(g_0+h\right)^{pq}(\nabla^{g_0}_ih_{pa}\nabla^{g_0}_jh_{qb}+2\nabla^{g_0}_ah_{jp}\nabla^{g_0}_qh_{ib}\\
&&-2\nabla^{g_0}_ah_{jp}\nabla^{g_0}_bh_{iq}-2\nabla^{g_0}_jh_{pa}\nabla^{g_0}_bh_{iq}-2\nabla^{g_0}_ih_{pa}\nabla^{g_0}_bh_{jq})\\
&&-\nabla^{g_0}_a\left(g_0+h\right)^{ab}\nabla^{g_0}_bh_{ij},\\
\nabla^{g_0}R_1[h]_{ij}&:=&\nabla^{g_0}_a\left(\left(\left(g_0+h\right)^{ab}-g_0^{ab}\right)\nabla^{g_0}_bh_{ij}\right).
\end{eqnarray*}
Note that one can simplify the zeroth order terms of $R_0[h]$ as follows :
\begin{eqnarray*}
&&-2\Rm(g_0)_{iklj}h_{kl}-g^{kl}g_{ip}\Rm(g_0)_{jklp}-g^{kl}g_{jp}\Rm(g_0)_{iklp}\\
&&+g_{ik}\Ric(g_0)_{kj}+g_{jk}\Ric(g_0)_{ki}\\
&=&-2\Rm(g_0)_{iklj}h_{kl}-g^{kl}h_{ip}\Rm(g_0)_{jklp}-g^{kl}h_{jp}\Rm(g_0)_{iklp}\\
&&-g^{kl}\Rm(g_0)_{jkli}-g^{kl}\Rm(g_0)_{iklj}\\
&&+g_{ik}\Ric(g_0)_{kj}+g_{jk}\Ric(g_0)_{ki}\\
&=&-2\Rm(g_0)_{iklj}h_{kl}-h_{ip}\Ric(g_0)_{jp}-h_{jp}\Ric(g_0)_{ip}\\
&&-h^{kl}h_{ip}\Rm(g_0)_{jklp}-h^{kl}h_{jp}\Rm(g_0)_{iklp}\\
&&-2\Ric(g_0)_{ij}-\Rm(g_0)_{jkli}h^{kl}-\Rm(g_0)_{iklj}h^{kl}\\
&&+2\Ric(g_0)_{ij}+h_{ik}\Ric(g_0)_{kj}+h_{jk}\Ric(g_0)_{ki}\\
&=&-2\Rm(g_0)_{iklj}(h_{kl}+h^{kl})\\
&&-h^{kl}h_{ip}\Rm(g_0)_{jklp}-h^{kl}h_{jp}\Rm(g_0)_{iklp}\\
&=&(g^{-1}\ast h\ast h\ast \Rm(g_0))_{ij},
\end{eqnarray*}
since, if $g_{kl}=(1+\lambda_k)\delta_{kl}$ is diagonalized in a orthonormal basis of $g_0$ at some point, where $\lambda_k\in\mathbb{R}$, for $k=1,...,n$,
\begin{eqnarray*}
h_{kl}+h^{kl}&=&\left(\lambda_k+(1+\lambda_k)^{-1}-1\right)\delta_{kl}\\
&=&(1+\lambda_k)^{-1}\lambda_k^2\delta_{kl}\\
&=&(g^{-1}\ast h\ast h)_{kl}.
\end{eqnarray*}
Therefore, $R_0[h]=g^{-1}\ast h\ast h\ast \Rm(g_0)+g^{-1}\ast g^{-1}\ast \nabla^{g_0}h\ast \nabla^{g_0}h.$\\

Now, we introduce the last flow equivalent to (\ref{Static-DeTurck-MRF}) that we will study in the next sections : 
\begin{eqnarray}\label{Time-Dep-DeTurck-MRF}
\left(\partial_t-L_t\right)\bar{h}= R_0[\bar{h}]+\nabla^{g_0(t)}R_1[\bar{h}],
\end{eqnarray}
with
\begin{eqnarray*}
R_{0}[\bar{h}]&:=&\bar{h}^{-1}\ast\bar{h}\ast\Rm(g_0(t))\\
&&+g(t)^{-1}\ast g(t)^{-1}\ast \nabla^{g_0(t)}\bar{h}\ast \nabla^{g_0(t)}\bar{h}\\
\nabla^{g_0(t)}R_1[\bar{h}]_{ij}&:=&\nabla^{g_0(t)}_a\left(\left(\left(g_0(t)+\bar{h}\right)^{ab}-g_0(t)^{ab}\right)\nabla^{g_0(t)}_b\bar{h}_{ij}\right),
\end{eqnarray*}
where,
\begin{eqnarray*}
&&L_t\bar{h}=\Delta_{g_0(t)}\bar{h}+2\Rm(g_0(t))\ast\bar{h}-\Ric(g_0(t))\otimes \bar{h}-\bar{h}\otimes \Ric(g_0(t)),\\
&& g_0(t):=(1+t)\phi_t^*g_0,\quad \bar{h}(t):=(1+t)\phi_t^*h(\ln(1+t)),\quad t\geq 0,\\
&&\partial_t\phi_t=-\frac{\nabla^{g_0}f_0}{1+t}\circ\phi_t.
\end{eqnarray*}
The flow (\ref{Time-Dep-DeTurck-MRF}) can be expressed globally by :
\begin{eqnarray*}
&&\partial_t\bar{g}=-2\Ric(\bar{g}(t))+\Li_{V(\bar{g}(t),g_0(t))}(g(t)),\\
&& \bar{g}(t):=g_0(t)+\bar{h}(t).
\end{eqnarray*}

\section{Function spaces}\label{sec-fct-spa}

Let $(M^n,g_0(t))_{t\geq 0}$ be a Ricci flow. Let $p\in (0,+\infty]$. For a family of symmetric $2$-tensors $(h(t))_{t\geq 0}\in S^2T^*M$, we define the average parabolic $L^p$-norm of $h$ as follows : 
\begin{eqnarray*}
&&\nor h\|_{L^p(P(x,R))}:=\left(\fint_{P(x,R)}\arrowvert h\arrowvert^p_{g_0(s)}(y,s)d\mu_{g_0(s)}(y)ds\right)^{1/p},\\
&&P(x,R):=\{(y,s)\in M\times \mathbb{R}_{+}^*\quad|\quad s\in (0,R^2],\quad y\in B_{g_0(s)}(x,R)\},
\end{eqnarray*}
 where $d\mu_{g_0(s)}(y)$ is the Riemannian measure associated to the metric $g_0(s)$ and where $B_{g_0(s)}(x,R)$ denotes the geodesic ball for the metric $g_0(s)$ centered at $x$ of radius $R$.

We define the function space $X$ as in \cite{Koc-Lam-Rou} : 
\begin{eqnarray*}
X&:=&\{h\quad|\quad\|h\|_{X}<+\infty\},\quad\mbox{where},\\
\|h\|_X&:=&\sup_{t\geq 0}\|h(t)\|_{L^{\infty}(M,g_0(t))}\\
&&+\sup_{(x,R)\in M\times \mathbb{R}_+^*}\left(R\nor  \nabla h\|_{L^2(P(x,R))}+\nor\sqrt{t}\nabla h\|_{L^{n+4}\left(P(x,R)\setminus P\left(x,\frac{R}{\sqrt{2}}\right)\right)}\right).
\end{eqnarray*}

We also consider the space $Y:=Y^0+\nabla Y^1,$ where 
\begin{eqnarray*}
 &&\|h\|_{Y_0}:=\sup_{(x,R)\in M\times\mathbb{R}_+^*}\left(R^2\nor h\|_{L^1(P(x,R))}+R^2\nor h\|_{L^{\frac{n+4}{2}}\left(P(x,R)\setminus P\left(x,\frac{R}{2}\right)\right)}\right),\\
 &&\|h\|_{Y_1}:=\sup_{(x,R)\in M\times\mathbb{R}_+^*}\left(R\nor h\|_{L^2(P(x,R))}+\nor \sqrt{t}h\|_{L^{n+4}\left(P(x,R)\setminus P\left(x,\frac{R}{2}\right)\right)}\right).\\
\end{eqnarray*}

These spaces are invariant under the following scaling, for some positive radius $R$ : 
\begin{eqnarray*}
g_{0,R}(t):=R^{-2}g_0(R^2t),\quad h_R(t):=R^{-2}h(R^2t),\quad t\geq 0,
\end{eqnarray*}
where $(g_0(t),h(t))$ is a solution to the Ricci-DeTurck flow (\ref{Time-Dep-DeTurck-MRF}).

As noticed in \cite{Koc-Lam-Rou}, we have the following estimate : 
\begin{lemma}\label{lemma-contraction-map}
Let $(M^n,g_0,\nabla^{g_0}f_0)$, $n\geq 3$, be an expanding gradient Ricci soliton with quadratic curvature decay. Then, for any $\gamma\in(0,1)$, the operator $R_0[\cdot]+\nabla R_1[\cdot]: B_X(0,\gamma)\subset X\rightarrow Y$ is analytic and satisfies
\begin{eqnarray*}
&&\|R_0[h]+\nabla R_1[h]\|_Y\leq c(\gamma,n,g_0)\| h\|^2_X,\quad h\in B_X(0,\gamma),\\
&&\\
&&\|R_0[h']-R_0[h]+\nabla (R_1[h']-R_1[h])\|_Y\leq c(\gamma,n,g_0)\left(\| h'\|_X+\| h \|_X\right)\| h'-h\|_X,
\end{eqnarray*}
for any  $h,h'\in B_X(0,\gamma)$.

\end{lemma}

\begin{proof}
This can be checked directly. The only non trivial estimate concerns the zeroth order quadratic term of $R_0[h]$ which uses the decay of the curvature tensor in an essential way. Indeed, as the metric $g_0$ has quadratic curvature decay, the metric $g_0(t)$ has quadratic curvature decay too and if $p$ is such that $\nabla^{g_0}f_0(p)=0$,
\begin{eqnarray*}
\arrowvert\Rm(g_0(t))\arrowvert_{g_0(t)}(x)\leq \frac{c}{1+t+d_{g_0(t)}^2(p,x)}, \quad t\geq 0,\quad x\in M,
\end{eqnarray*}
for some positive constant $c$ independent of time. Therefore,
\begin{eqnarray*}
\fint_{B_{g_0(t)}(x,R)}\arrowvert\Rm(g_0(t))\arrowvert_{g_0(t)}d\mu_{g_0(t)}\leq c\fint_{B_{g_0(t)}(x,R)} \frac{1}{1+t+d_{g_0(t)}^2(p,y)}   d\mu_{g_0(t)},
\end{eqnarray*}
for any nonnegative time $t$. If $d_{g_0(t)}(p,x)\geq 2R$, then $d_{g_0(t)}(p,y)\geq R$ for any $y\in B_{g_0(t)}(x,R)$ and 
\begin{eqnarray*}
\fint_{B_{g_0(t)}(x,R)} \frac{1}{1+t+d_{g_0(t)}^2(p,y)}d\mu_{g_0(t)}\leq \frac{c}{1+R^2},
\end{eqnarray*}
for some positive constant $c$ uniform in time, space and radius $R$. If $d_{g_0(t)}(p,x)\leq 2R$ then, by the co-area formula,
\begin{eqnarray*}
\int_{B_{g_0(t)}(x,R)} \frac{1}{1+t+d_{g_0(t)}^2(p,y)}d\mu_{g_0(t)}&\leq& \int_{B_{g_0(t)}(p,3R)}\frac{1}{1+t+d_{g_0(t)}^2(p,y)}d\mu_{g_0(t)}\\
&\leq&c(n)\int_0^{3R}\frac{r^{n-1}}{1+r^2}dr\\
&\leq& c(n) R^{n-2},
\end{eqnarray*}
if $n\geq 3$. In any case, we get : 
\begin{eqnarray*}
\int_0^{R^2}\fint_{B_{g_0(t)}(x,R)}\arrowvert\Rm(g_0(t))\arrowvert_{g_0(t)}d\mu_{g_0(t)}\leq c,
\end{eqnarray*}
for some positive constant $c$ uniform in time, space and radius $R$. 

\end{proof}

\section{Estimates for the homogeneous linear equation}\label{sec-est-hom-equ}
\begin{theo}\label{Est-Hom-lin-equ}
Let $(M^n,g_0(t))_{t\geq 0}$, $n\geq 3$, be an expanding gradient Ricci soliton with positive curvature operator and with quadratic curvature decay, i.e. such that $$\limsup_{x\rightarrow+\infty}d_{g_0}(p,x)^2\R_{g_0}(x)<+\infty,$$ for some (any) point $p\in M$. Let $(h(t))_{t\geq 0}$ be a solution to the homogeneous linear equation $\partial_th=L_th$ with $h_0=h(0)\in L^{\infty}(S^2T^*M,g_0)$. Then, $(h(t))_{t\geq 0}\in X$, and
\begin{eqnarray*}
\| h\|_X\leq c\| h_0\|_{L^{\infty}(M,g_0)},
\end{eqnarray*}
for some uniform positive constant $c$.\\
\end{theo}

\begin{rk}
The proof of theorem \ref{Est-Hom-lin-equ} actually shows that the heat kernel associated to the Lichnerowicz operator is positivity improving in case $(M^n,g_0(t))_{t\geq 0}$ is an expanding gradient Ricci soliton with positive curvature operator, i.e. with the notations of theorem \ref{Est-Hom-lin-equ},
\begin{eqnarray*}
h_0\geq 0,\quad h_0\in S^2T^*M\setminus\{0\}\Longrightarrow h(t)>0,\quad t>0.\\
\end{eqnarray*}

\end{rk}

\begin{proof}
\begin{itemize}
\item $L^{\infty}$ estimate :\\

First of all, let us notice that $H(t):=g_0(t)+2(1+t)\Ric(g_0(t))$ satisfies $\partial_tH=L_tH$, $H(0)=g_0+2\Ric(g_0)(=\Li_{\nabla^{g_0}f_0}(g_0)).$ 

Indeed, this comes from the well-known fact that the Ricci curvature $\Ric(g_0(t))$ satisfies $\partial_t\Ric(g_0(t))=L_t\Ric(g_0(t))$ along a Ricci flow $(g_0(t))_{t\geq 0}$. On the other hand, there is some constant $c$ such that 
\begin{eqnarray*}
h(0)-c\arrowvert h_0\arrowvert_{L^{\infty}(M,g_0)}H(0)=h_0-c\arrowvert h_0\arrowvert_{L^{\infty}(M,g_0)}(g_0+2\Ric(g_0))\geq 0,\quad \mbox{on $M$},
\end{eqnarray*}
since $H(0)$ is positive definite (since $\Ric(g_0)\geq 0$).

We claim that for any positive time $t$, $h(t)-c\arrowvert h_0\arrowvert_{L^{\infty}(M,g_0)}H(t)\geq 0$. To prove this claim, define for any positive $\epsilon$,
\begin{eqnarray*}
h_{\epsilon}(t):=h(t)-c\arrowvert h_0\arrowvert_{L^{\infty}(M,g_0)}H(t)+\epsilon e^{Ct}\phi(t) H(t),
\end{eqnarray*}
for some positive constant $C$ to be chosen later and $\phi:M\times \mathbb{R}_+^*\rightarrow\mathbb{R}_+$ is a smooth function defined  by :
\begin{eqnarray}\label{function-phi-ad-hoc}
\phi(x,t):=\ln\left((1+t)\phi_t^*\left(f+\mu(g)+\frac{n}{2}\right)\right),
\end{eqnarray}
where $\mu(g)$ defined the entropy (see Appendix \ref{sol-equ-sec}) and where $(\phi_t)_t$ is the flow generated by $-\nabla^{g_0}f_0/(1+t)$. It is well defined since by the soliton equations, $f+\mu(g)+n/2> n/2\geq1$. In particular, $\phi$ satisfies on $M\times [0,T]$ for some positive finite time $T$,
\begin{eqnarray*}
\arrowvert\partial_t\phi\arrowvert+\arrowvert\nabla^{g_0(t)}\phi\arrowvert+\arrowvert\Delta_{g_0(t)}\phi\arrowvert\leq C(T),
\end{eqnarray*}
  for some positive space independent constant $C(T)$ depending on $T$ (at least). 

Therefore,
\begin{eqnarray*}
\partial_th_{\epsilon}&=&L_th_{\epsilon}+\epsilon e^{Ct}\left(C\phi H(t)-2\nabla^{g_0(t)}_{\nabla^{g_0(t)} \phi}H(t)-\Delta_{g_0(t)}\phi H(t)\right)\\
&\geq&L_th_{\epsilon}+\epsilon e^{Ct}\left(C\inf_{M\times [0,T]}\phi-C(T)\right)H(t),
\end{eqnarray*}
since $H(t)$ and $g(t)$ define equivalent metrics. Since $\inf_{M\times [0,T]}\phi$ is positive, one can choose $C$ big enough so that,
\begin{eqnarray*}
\partial_th_{\epsilon}&>&L_th_{\epsilon},
\end{eqnarray*}
on $M\times [0,T]$. 

Now, if $\inf_{M\times[0,T]}h_{\epsilon}<0$, there must be a first space-time point $(x_0,t_0)\in M\times[0,T]$ and some unitary vector $u\in T_xM$ such that  $h_{\epsilon}(x_0,t_0)(u)=0$ since $\phi$ is an exhaustion function (Appendix \ref{sol-equ-sec}). In particular, $h_{\epsilon}(t)\geq 0$ for  $t\leq t_0$. Extend $u$ in a neighborhood of $x_0$ by parallel transport with respect to $g_0(t_0)$ independently of time as done in [Theorem $3.3$, Chap. $2$,\cite{Ben}]. This vector field will be denoted by $U$. Then, as $t_0$ is positive, we have at $(x_0,t_0)$,
\begin{eqnarray*}
0\geq\partial_t\left(h_{\epsilon}(U,U)\right)&=&\left(\partial_th_{\epsilon}\right)(U,U)\\
&>&L_{t_0}h_{\epsilon}(U,U)\\
&=&\Delta_{g_0(t_0)}\left(h_{\epsilon}(U,U)\right)+(2\Rm(g_0(t_0))\ast h_{\epsilon})(U,U)\\
&\geq&\Delta_{g_0(t_0)}\left(h_{\epsilon}(U,U)\right)\\
&\geq&0,
\end{eqnarray*}
which is a contradiction. Therefore, $h_{\epsilon}\geq 0$ on $M\times [0,T]$ for any positive $\epsilon$, i.e. $h(t)\geq c\arrowvert h_0\arrowvert_{L^{\infty}(M,g_0)} (g_0+2\Ric(g_0))$. The same previous argument can be applied to $(-h(t))_{t\geq 0}$ which shows $\sup_{M\times \mathbb{R}_+}\|h(t)\|_{L^{\infty}(M,g_0(t))}\leq \arrowvert c\arrowvert(1+2\sup_M\Ric(g_0))\| h_0\|_{L^{\infty}(M,g_0)}.$\\

\item $L^2$ estimate : \\

Since $(h(t))_{t\geq 0}$ satisfies $\partial_th=L_th$, one has,
\begin{eqnarray*}
\partial_t\arrowvert h\arrowvert_{g_0(t)}^2\leq \Delta_{g_0(t)}\arrowvert h\arrowvert^2_{g_0(t)}-2\arrowvert\nabla^{g_0(t)}h\arrowvert_{g_0(t)}^2+c(n)\arrowvert\Rm(g_0(t))\arrowvert_{g_0(t)}\arrowvert h\arrowvert_{g_0(t)}^2.
\end{eqnarray*}
Multiplying this inequality by any smooth cut-off function $\psi^2 : M\times\mathbb{R}_+\rightarrow \mathbb{R}_+$ gives, after integrating by parts (in space) : 
\begin{eqnarray*}
&&\partial_t\int_M\psi^2\arrowvert h\arrowvert^2_{g_0(t)}d\mu_{g_0(t)}+2\int_M\psi^2\arrowvert \nabla^{g_0(t)}h\arrowvert_{g_0(t)}^2d\mu_{g_0(t)}\leq\\
&& \int_M\left<-\nabla^{g_0(t)}\psi^2,\nabla^{g_0(t)}\arrowvert h\arrowvert^2_{g_0(t)}\right>+\left(\partial_t\psi^2+c(n)\psi^2\arrowvert\Rm(g_0(t))\arrowvert_{g_0(t)}\right)\arrowvert h\arrowvert^2_{g_0(t)}d\mu_{g_0(t)}\\
&&\leq \int_M4\arrowvert\nabla^{g_0(t)}\psi\arrowvert_{g_0(t)}\arrowvert h\arrowvert_{g_0(t)}\psi\arrowvert\nabla^{g_0(t)}h\arrowvert_{g_0(t)}+\left(\partial_t\psi^2+c(n)\psi^2\arrowvert\Rm(g_0(t))\arrowvert_{g_0(t)}\right)\arrowvert h\arrowvert^2_{g_0(t)}d\mu_{g_0(t)}.
\end{eqnarray*}
 Hence, by applying Young's inequality to the term $4\arrowvert\nabla^{g_0(t)}\psi\arrowvert_{g_0(t)}\arrowvert h\arrowvert_{g_0(t)}\psi\arrowvert\nabla^{g_0(t)}h\arrowvert_{g_0(t)}$,

\begin{eqnarray*}
&&\partial_t\int_M\psi^2\arrowvert h\arrowvert^2_{g_0(t)}d\mu_{g_0(t)}+\int_M\psi^2\arrowvert \nabla^{g_0(t)}h\arrowvert_{g_0(t)}^2d\mu_{g_0(t)}\leq\\
&& c(n)\int_M\left(\arrowvert\nabla^{g_0(t)}\psi\arrowvert^2_{g_0(t)}+\partial_t\psi^2+\psi^2\arrowvert\Rm(g_0(t))\arrowvert_{g_0(t)}\right)\arrowvert h\arrowvert^2_{g_0(t)}d\mu_{g_0(t)}.
\end{eqnarray*}

Now, let $(x,R)\in M\times\mathbb{R}_+^*$ and consider the following cutoff function : $\psi_{x,R}(y,t):=\psi(d_{g_0(t)}(x,y)/R),$ where $\psi:\mathbb{R}_+\rightarrow \mathbb{R}_+ $ is a smooth function such that $\psi_{|[0,1]}\equiv 1$, $\psi_{|[2,+\infty)}\equiv 0$ and $\sup_{\mathbb{R}_+}\arrowvert\psi'\arrowvert\leq c$. $\psi_{x,R}$ is a Lipschitz function satisfying :
\begin{eqnarray*}
&&\arrowvert\nabla^{g_0(t)}\psi_{x,R}\arrowvert_{g_0(t)}\leq \frac{c}{R},\quad0\leq\partial_t\psi_{x,R}\leq \frac{c}{R\sqrt{t}},
\end{eqnarray*}
almost everywhere, by lemma \ref{C^0-distance-est-Type-III}. Hence, by integrating in time : 
\begin{eqnarray*}
&&\int_{B_{g_0(R^2)}(x,R)}\arrowvert h\arrowvert^2_{g_0(R^2)}d\mu_{g_0(R^2)}+\int_0^{R^2}\int_{B_{g_0(t)}(x,R)}\arrowvert\nabla^{g_0(t)}h\arrowvert^2_{g_0(t)}d\mu_{g_0(t)}dt\leq\\
&&\int_{B_{g_0(0)}(x,2R)}\arrowvert h(0)\arrowvert^2_{g_0(0)}d\mu_{g_0(0)}\\
&&+c\int_0^{R^2}\int_{B_{g_0(t)}(x,2R)}\left(\frac{1}{R^2}+\frac{1}{R\sqrt{t}}+\arrowvert\Rm(g_0(t))\arrowvert_{g_0(t)}\right)\arrowvert h\arrowvert^2_{g_0(t)}d\mu_{g_0(t)}dt\\
&\leq&c\| h(0)\|^2_{L^{\infty}(M,g_0)}\left(R^n +\int_0^{R^2}\int_{B_{g_0(t)}(x,2R)}\arrowvert\Rm(g_0(t))\arrowvert_{g_0(t)}\right)d\mu_{g_0(t)}dt,
\end{eqnarray*}
since $\vol_{g_0(t)}B_{g_0(t)}(x,s)\leq c(n)s^n$ for any nonnegative $s$ by the Bishop-Gromov theorem since $\Ric(g_0(t))\geq 0$. Now, as $\AVR(g_0(t))=\AVR(g_0)>0$ by Appendix \ref{sol-equ-sec}, we get the reverse uniform inequality $\vol_{g_0(t)}B_{g_0(t)}(x,s)\geq \AVR(g_0)s^n$, for any nonnegative $s$ by the Bishop-Gromov theorem again. Hence,
\begin{eqnarray*}
R^2\nor h\|_{L^2(P(x,R))}^2\leq c\| h(0)\|^2_{L^{\infty}(M,g_0)}\left(1+\int_0^{R^2}\fint_{B_{g_0(t)}(x,2R)}\arrowvert\Rm(g_0(t))\arrowvert_{g_0(t)}d\mu_{g_0(t)}dt\right).
\end{eqnarray*}

Hence the desired $L^2$ estimate by the proof of lemma \ref{lemma-contraction-map}. \\

\item $L^{n+4}$ estimate : \\
We prove that 
\begin{eqnarray}\label{grad-est-resc}
\sup_{(x,R)\in M\times(0,+\infty)}\|\sqrt{t}\nabla^{g_0(t)} h\|_{L^{\infty}\left(P(x,R)\setminus P(x,R/\sqrt{2})\right)}\leq c\| h\|_{L^{\infty}(M,g_0)}.
\end{eqnarray}
We establish rough a priori estimates by computing the evolution of $$\arrowvert h\arrowvert_{g_0(t)}^2+t\arrowvert\nabla^{g_0(t)}h\arrowvert_{g_0(t)}^2,$$ as follows : 
\begin{eqnarray*}
\partial_t\arrowvert h\arrowvert_{g_0(t)}^2&\leq& \Delta_{g_0(t)}\arrowvert h\arrowvert_{g_0(t)}^2-2\arrowvert\nabla^{g_0(t)} h\arrowvert_{g_0(t)}^2+c(n)\arrowvert \Rm(g_0(t))\arrowvert_{g_0(t)}\arrowvert h\arrowvert_{g_0(t)}^2,\\
\partial_t\left(t\arrowvert\nabla^{g_0(t)} h\arrowvert_{g_0(t)}^2\right)&\leq&\Delta_{g_0(t)}\left(t\arrowvert\nabla^{g_0(t)} h\arrowvert_{g_0(t)}^2\right)-2t\arrowvert\nabla^{g_0(t),2}h\arrowvert_{g_0(t)}^2+\arrowvert\nabla^{g_0(t)} h\arrowvert_{g_0(t)}^2\\
&&+c(n)\arrowvert\Rm(g_0(t))\arrowvert_{g_0(t)}t\arrowvert\nabla^{g_0(t)} h\arrowvert_{g_0(t)}^2\\
&&+c(n)\sqrt{t}\arrowvert\nabla^{g_0(t)}\Rm(g_0(t))\arrowvert_{g_0(t)}\sqrt{t}\arrowvert\nabla^{g_0(t)} h\arrowvert_{g_0(t)}\arrowvert h\arrowvert_{g_0(t)}.
\end{eqnarray*}
Therefore, as $(M^n,g_0(t))_{t\geq0}$ is a Type III solution of the Ricci flow, i.e. 
\begin{eqnarray*}
\arrowvert\Rm(g_0(t))\arrowvert_{g_0(t)}+\sqrt{t}\arrowvert\nabla^{g_0(t)}\Rm(g_0(t))\arrowvert_{g_0(t)}\leq \frac{c}{1+t},
\end{eqnarray*}
 for $t\geq 0$, one has
\begin{eqnarray*}
\left(\partial_t-\Delta_{g_0(t)}\right)\left(\arrowvert h\arrowvert_{g_0(t)}^2+t\arrowvert\nabla^{g_0(t)} h\arrowvert_{g_0(t)}^2\right)&\leq&\frac{c}{1+t}\left(\arrowvert h\arrowvert_{g_0(t)}^2+t\arrowvert\nabla^{g_0(t)} h\arrowvert_{g_0(t)}^2\right).
\end{eqnarray*}
In particular,
\begin{eqnarray}\label{a-priori-grad-linear-est}
\left(\partial_t-\Delta_{g_0(t)}\right)\left((1+t)^{-c}\left(\arrowvert h\arrowvert_{g_0(t)}^2+t\arrowvert\nabla^{g_0(t)} h\arrowvert_{g_0(t)}^2\right)\right)\leq 0,\quad t\geq 0.
\end{eqnarray}

To end the argument, we need to prove a mean value inequality in the spirit of Zhang \cite{Zha-Qi-Gra-Est} in a Ricci flow setting. The proof is very close to his, but we reproduce it here for the convenience of the reader and for the consistency of the paper. We must be careful about the time dependence of the constants involved in this mean value inequality.

\begin{prop}[$L^1$ mean value inequality]\label{L^1-mean-inequ}
Let $(M^n,g_0(t))_{t\geq 0}$ be a Type III solution of the Ricci flow with nonnegative Ricci curvature which is noncollapsed, i.e. 
\begin{eqnarray*}
(1+t)\arrowvert\Rm(g_0(t))\arrowvert_{g_0(t)}\leq R_0,\quad\Ric(g_0(t))\geq 0,\quad \AVR(g_0(t))\geq V_0>0,\quad  t\geq0. 
\end{eqnarray*}

 Then any nonnegative subsolution $u$ of the heat equation, i.e.
\begin{eqnarray*}
\partial_tu\leq \Delta_{g_0(t)}u,\quad \mbox{on $M\times (0,+\infty)$,}
\end{eqnarray*}
satisfies, for $r^2< t$, $\theta\in(0,1)$,
\begin{eqnarray*}
\sup_{P\left(x,t,\theta r\right)}u\leq \frac{c(n,V_0,R_0,\theta)}{r^{n+2}}\int_{P(x,t,r)}ud\mu_{g_0(s)}ds,
\end{eqnarray*}
where $P(x,t,r):=\{(y,s)\in M\times[0,+\infty)\quad|\quad s\in (t-r^2,t],\quad y\in B_{g_0(s)}(x,r)\}.$
\end{prop}
\begin{proof}[Proof of proposition \ref{L^1-mean-inequ}]
Let $p\in [1,+\infty)$. Then, 
\begin{eqnarray}\label{sub-sol-power}
\partial_tu^p\leq \Delta_{g_0(t)}u^p,
\end{eqnarray}
 on $M\times (0,+\infty)$. Take any smooth space-time cutoff function $\psi$ and multiply (\ref{sub-sol-power}) by $\psi^2u^p$ and integrate by parts as follows : 
 \begin{eqnarray*}
&&-\int_{t- r^2}^{t'}\int_M\psi^2u^p\Delta_{g_0(s)}u^pd\mu_{g_0(s)}ds\leq -\int_{t-r^2}^{t'}\int_M\psi^2u^p\partial_su^pd\mu_{g_0(s)}ds,\\
&&\int_{t- r^2}^{t'}\int_M\arrowvert\nabla^{g_0(s)}(\psi u^p)\arrowvert^2_{g_0(s)}-\arrowvert\nabla^{g_0(s)}\psi\arrowvert^2_{g_0(s)}u^{2p} d\mu_{g_0(s)}ds\leq\\
&& \int_{t-r^2}^{t'}\int_M\frac{\partial_s(\psi^2\ln d\mu_{g_0(s)})}{2}u^{2p}d\mu_{g_0(s)}ds-\frac{1}{2}\int_M\psi^2u^{2p}d\mu_{g_0(t)},
\end{eqnarray*}
for any $t'\in(t-r^2,t].$
 Since the scalar curvature is nonnegative, $\partial_s(\ln d\mu_s)=-\R_{g_0(s)}\leq 0$, for $s\geq 0$. Hence,
\begin{eqnarray*}
&&\int_{t-r^2}^{t'}\int_M\arrowvert\nabla^{g_0(s)}(\psi u^p)\arrowvert^2_{g_0(s)}d\mu_{g_0(s)}ds+\frac{1}{2}\int_M\psi^2u^{2p}d\mu_{g_0(t)}\leq\\
&&\int_{t-r^2}^{t'}\int_M\left(\frac{\partial_s\psi^2}{2}+\arrowvert\nabla^{g_0(s)}\psi\arrowvert_{g_0(s)}^2\right)u^{2p}d\mu_{g_0(s)}ds.
\end{eqnarray*}
Let $\tau, \sigma\in (0,+\infty)$ such that $\tau+\sigma<r$. In particular, $P(x, \tau,\sigma)\subset P(x,t,\tau+\sigma)\subset P(x,t,r)$.
Now, choose two smooth functions $\phi:\mathbb{R}_+\rightarrow[0,1]$ and $\eta:\mathbb{R}_+\rightarrow[0,1]$ such that
\begin{eqnarray*}
&& \supp(\phi)\subset [\tau,\tau+\sigma],\quad\phi\equiv 1\quad\mbox{in $[0,\tau]$},\quad \phi\equiv 0\quad\mbox{in $[\tau+\sigma,+\infty)$},\quad -c/\sigma\leq \phi'\leq 0,\\
&& \supp(\eta)\subset [t-(\tau+\sigma)^2,+\infty),\quad\eta\equiv 1\quad\mbox{in $[t-\tau^2,+\infty)$},\\
&& \eta\equiv 0\quad\mbox{in $(t-r^2,t-(\tau+\sigma)^2]$},\quad 0\leq \eta'\leq c/\sigma^2.
\end{eqnarray*}
Define $\psi(y,s):=\phi(d_{g_0(s)}(x,y))\eta(s)$, for $(y,s)\in M\times(0,+\infty)$. Then, thanks to lemma \ref{C^0-distance-est-Type-III},
\begin{eqnarray*}
&&\arrowvert\nabla^{g_0(s)}\psi\arrowvert_{g_0(s)}\leq\frac{c(R_0)}{\sigma},\quad \arrowvert\partial_s\psi\arrowvert\leq \frac{c}{\sigma^2}+\frac{c(R_0)}{\sigma{\sqrt{t-(\tau+\sigma)^2}}},
\end{eqnarray*}
for some uniform positive constant $c(R_0)$.
On the other hand, $(M^n,g_0(s))_{s\geq 0}$ satisfies the following Euclidean Sobolev inequality [Chap.$3$, \cite{Sal-Cos-Sob-Boo}] since it has nonnegative Ricci curvature and it is non collapsed :
\begin{eqnarray*}
\left(\int_M(\psi u^p)^{\frac{2n}{n-2}}d\mu_{g_0(s)}\right)^{\frac{n-2}{n}}\leq c(n,V_0)\int_M\arrowvert\nabla^{g_0(s)}(\psi u^p)\arrowvert_{g_0(s)}^2d\mu_{g_0(s)},
\end{eqnarray*}
for any $s\geq 0$.
We could have used a local Sobolev inequality which would amount to the same by using the non collapsing assumption.

Now, by Hölder inequality, for $s\geq 0$,
\begin{eqnarray*}
\int_M (\psi u^p)^{2+\frac{4}{n}}d\mu_{g_0(s)}\leq\left(\int_M (\psi u^p)^{\frac{2n}{n-2}}d\mu_{g_0(s)}\right)^{\frac{n-2}{n}}\left(\int_M (\psi u^p)^2d\mu_{g_0(s)}\right)^{\frac{2}{n}}.
\end{eqnarray*}
Therefore, to sum it up, if $\alpha_n:=1+2/n$,
\begin{eqnarray*}
\int_{P(x,t,\tau)}\left(u^{2p}\right)^{\alpha_n}d\mu_{g_0(s)}ds&\leq&\int_{P(x,t,r)}\left(\psi u^p\right)^{2\alpha_n}d\mu_{g_0(s)}ds\\
&\leq&\int_{t-r^2}^t\left(\int_M (\psi u^p)^{\frac{2n}{n-2}}d\mu_{g_0(s)}\right)^{\frac{n-2}{n}}\left(\int_M (\psi u^p)^2d\mu_{g_0(s)}\right)^{\frac{2}{n}}ds\\
&\leq&c(n,V_0,R_0)\sup_{s\in(t-r^2,t]}\left(\int_M (\psi u^p)^2d\mu_{g_0(s)}\right)^{\frac{2}{n}}\int_{P(x,t,r)}\arrowvert\nabla^{g_0(s)}(\psi u^p)\arrowvert_{g_0(s)}^2d\mu_{g_0(s)}\\
&\leq&c(n,V_0,R_0)\left(\frac{1}{\sigma^2}+\frac{1}{\sigma{\sqrt{t-(\tau+\sigma)^2}}}\right)\left(\int_{P(x,t,\tau+\sigma)}u^{2p}d\mu_{g_0(s)}\right)^{\alpha_n}.
\end{eqnarray*}

Define the following sequences : 
\begin{eqnarray*}
p_i:=\alpha_n^i,\quad \sigma_i:=2^{-1-i}(1-\theta)r,\quad \tau_{-1}:=r,\quad\tau_i:=r-\sum_{j=0}^i\sigma_j,\quad i\geq 0.
\end{eqnarray*}
Then, $\lim_{i\rightarrow +\infty}\tau_i=\theta r$ and, for any $i\geq 0$,
\begin{eqnarray*}
\| u^2\|_{L^{p_{i+1}}(P(x,t,\tau_i))}\leq\left(c(n,V_0,R_0)\left(\frac{1}{\sigma_i^2}+\frac{1}{\sigma_i{\sqrt{t-\tau_{i-1}^2}}}\right)\right)^{\frac{1}{\alpha_n}}\| u^2\|_{L^{p_i}(P(x,t,\tau_{i-1}))},
\end{eqnarray*}
i.e.
\begin{eqnarray*}
\| u\|^2_{L^{\infty}(P(x,t,\theta r))}\leq\Pi_{i=0}^{\infty}\left(c(n,V_0,R_0)\left(\frac{1}{\sigma_i^2}+\frac{1}{\sigma_i{\sqrt{t-\tau_{i-1}^2}}}\right)\right)^{\frac{1}{\alpha_n}}\| u\|^2_{L^{2}(P(x,t,r))}.
\end{eqnarray*}
It remains to estimate the previous infinite product. It turns out that 
\begin{eqnarray*}
&&\sqrt{t-\tau_{i-1}^2}\geq\sqrt{t-(\theta r)^2}\geq c(\theta)\sigma_0\geq c(\theta)\sigma_i,\quad i\geq 0,\\
&& c(\theta)= 2\left(\frac{1+\theta}{1-\theta}\right)^{1/2}.
\end{eqnarray*}
Therefore,
\begin{eqnarray*}
\Pi_{i=0}^{\infty}\left(\frac{1}{\sigma_i^2}+\frac{1}{\sigma_i{\sqrt{t-\tau_{i-1}^2}}}\right)^{\frac{1}{\alpha_n}}&\leq& \max\left\{1,\frac{1}{c(\theta)}\right\}^{1+n/2}\Pi_{i=0}^{+\infty}\left(4^{i+1}(1-\theta)^{-2}r^{-2}\right)^{\alpha_n^{-i}}\\
&\leq&c(n)\max\left\{1,\frac{1}{c(\theta)}\right\}^{1+n/2}\frac{1}{((1-\theta)r)^{n+2}}\\
&\leq&c'(n)\frac{1}{((1-\theta)r)^{n+2}},
\end{eqnarray*}
i.e.
\begin{eqnarray}\label{L^2-bound-mean-value}
\sup_{P(x,t,\theta r)}u^2\leq \frac{c(n,V_0,R_0)}{((1-\theta)r)^{n+2}}\int_{P(x,t,r)}u^2d\mu_{g_0(s)}ds.
\end{eqnarray}

To get a bound depending on the $L^1$ norm of $u$, we proceed as in \cite{Li-Sch-Mea-Val}. Define $U_k:=\sup_{P(x,t,r_k)}u^2$, where $$r_k:=r(1-\theta)\sum_{i=0}^{k}\theta^{i+1},$$ for $k\geq 0$ . Then, by (\ref{L^2-bound-mean-value}), one has, for any $k\geq 0$,
\begin{eqnarray*}
U_k\leq\left(\frac{c(n,V_0,R_0)\theta^{-k-1}}{r-\theta r}\right)^{n+2}\left(\int_{P(x,t,r/2)}ud\mu_{g_0(s)}ds\right)(U_{k+1})^{1/2}.
\end{eqnarray*}
Hence the result.

\end{proof}

We end the proof of the $L^{n+4}$ estimate as follows. By applying lemma \ref{L^1-mean-inequ} for $t:=R^2$, $r:=(3/4)^{1/2}R$, and $\theta:=(2/3)^{1/2}$ for some positive radius $R$ to the subsolution $(1+t)^{-c}(\arrowvert h\arrowvert_{g_0(t)}^2+t\arrowvert\nabla^{g_0(t)}h\arrowvert^2_{g_0(t)})$ (see (\ref{a-priori-grad-linear-est})), one has : 
\begin{eqnarray*}
&&P(x,t,\theta r)=P(x,R)\setminus P(x,R/\sqrt{2}),\\
&&\sup_{P(x,t,\theta r)}(1+s)^{-c}(\arrowvert h\arrowvert_{g_0(s)}^2+s\arrowvert\nabla^{g_0(s)}h\arrowvert^2_{g_0(s)})\leq \\
&& c(n,V_0,R_0,\theta)\fint_{P(x,t,r)} (1+s)^{-c}(\arrowvert h\arrowvert_{g_0(s)}^2+s\arrowvert\nabla^{g_0(s)}h\arrowvert^2_{g_0(s)})d\mu_{g_0(s)}ds.
\end{eqnarray*}
Now, by the $L^{\infty}$ bound together with the $L^2$ bound, one has 
\begin{eqnarray*}
\sup_{P(x,t,\theta r)}(\arrowvert h\arrowvert_{g_0(s)}^2+s\arrowvert\nabla^{g_0(s)}h\arrowvert^2_{g_0(s)})\leq c(n,V_0,R_0,\theta)\arrowvert h_0\arrowvert^2_{L^{\infty}(M,g_0)}.
\end{eqnarray*}

Hence we get (\ref{grad-est-resc}), hence the $L^{n+4}$ bound.

\end{itemize}
\end{proof}

\begin{rk}
The $L^{\infty}$ estimate of Theorem \ref{Est-Hom-lin-equ} shows that the operator $h_s\in L^{\infty}\rightarrow h_t\in L^{\infty}$ for $t>s\geq 0$ where $(h_t)_{t>s}$ is a solution to the homogeneous Licherowicz heat equation is contractive for the scalar product $H(t):=g(t)+2(1+t)\Ric(g(t))$. 
\end{rk}
We close this section by estimating the heat kernel associated to the Lichnerowicz operator in terms of the heat kernel associated to the laplacian acting on functions :
\begin{theo}\label{comp-heat-ker}
Let $(M^n,g_0(t))_{t\geq 0}$, $n\geq 3$, be an expanding gradient Ricci soliton with positive curvature operator. Let $(h(t))_{t\geq 0}$ be a solution to the homogeneous linear equation $\partial_th=L_th$ with $h_0=h(0)\in L^{\infty}(S^2T^*M,g_0)$. Let $(u(t))_{t\geq 0}$ be a solution to the homogeneous linear equation $\partial_tu=\Delta_{g_0(t)}u+\R_{g_0(t)}u$ with $u(0)=\arrowvert h_0\arrowvert_{g_0}\in L^{\infty}(M,g_0)$. Then, 
\begin{eqnarray*}
-u(t)g_0(t)\leq h(t)\leq u(t)g_0(t),\quad  t\geq 0,
\end{eqnarray*}
in the sense of quadratic forms.
\end{theo}

\begin{proof}
The proof of theorem \ref{comp-heat-ker} mimics the proof of the $L^{\infty}$ estimate of theorem \ref{Est-Hom-lin-equ}. Indeed, consider the following symmetric $2$-tensor 
\begin{eqnarray*}
\bar{h}(t)&:=&h(t)+u(t)g(t),\quad t\geq 0.\\
\end{eqnarray*}
Then, one computes the evolution of $(\bar{h}(t))_{t\geq 0}$ as follows :
\begin{eqnarray*}
(\partial_t-L_t)\bar{h}&=&\R_{g_0(t)}u(t)g_0(t)-2u(t)\Ric(g_0(t))\\
&\geq&\R_{g_0(t)}u(t)g_0(t)-u(t)\R_{g_0(t)}g_0(t)\\
&\geq&0,
\end{eqnarray*}
since $g_0(t)$ has positive curvature operator. Now, by using a barrier function as in the proof of the $L^{\infty}$ estimate of theorem \ref{Est-Hom-lin-equ}, the result follows by applying the maximum principle for symmetric $2$-tensors.

\end{proof}

\begin{rk}
The proof of theorem \ref{comp-heat-ker} does not use the expanding structure in an essential way.
\end{rk}

\section{Heat kernel estimates}\label{sec-hea-ker-est}

\subsection{Estimates of the heat kernel acting on functions}
We follow the same strategy adopted by Grigor'yan and Zhang.

Let $(M^n,g(t))_{t\in[0,T)}$ be a complete Ricci flow with bounded curvature for each time slice. We consider the scalar heat equation with potential coupled with this Ricci flow : 
\begin{equation}
\label{scal-heat-equ}
\left\{
\begin{array}{rl}
&\partial_tu=\Delta_{g(t)}u+\R_{g(t)}u,\\
&\\
& \partial_tg=-2\Ric(g(t)),
\end{array}
\right.
\end{equation}
on $M^n\times(0,T)$.
We also consider its conjugate heat equation : 

\begin{equation}
\label{scal-heat-equ-conj}
\left\{
\begin{array}{rl}
&\partial_{\tau}u=\Delta_{g(\tau)}u,\\
&\\
&\partial_{\tau}g=2\Ric(g(\tau)),\\
\end{array}
\right.
\end{equation}
on $M\times (0,t]$ where $\tau(s):=t-s$ for some fixed positive time $t$.

We denote the heat kernel associated to (\ref{scal-heat-equ}) by $K(x,t,y,s)$, for $0\leq s<t<T$, and $x,y\in M$. This heat kernel always exists and is positive : see \cite{Gue-Fun-Sol}. It satisfies by definition, for any fixed $(y,s)\in M\times[0,T)$,

\begin{equation}
\label{scal-heat-equ-ker}
\left\{
\begin{array}{rl}
&\partial_tK(\cdot,\cdot,y,s)=\Delta_{g(t)}K(\cdot,\cdot,y,s)+\R_{g(t)}K(\cdot,\cdot,y,s),\\
&\\
&\partial_tg=-2\Ric(g(t)),\\
&\\
&\lim_{t\rightarrow s}K(\cdot,t,y,s)=\delta_y.\\
\end{array}
\right.
\end{equation}

On the other hand, if $(x,t)\in M\times(0,T)$ is fixed, then $K(x,t,\cdot,\cdot)$ is the heat kernel associated to the conjugate backward heat equation : 
\begin{equation}
\label{scal-heat-equ-conj-ker}
\left\{
\begin{array}{rl}
&\partial_sK(x,t,\cdot,\cdot)=-\Delta_{g(s)}K(x,t,\cdot,\cdot),\\
&\\
&\partial_sg=-2\Ric(g(s)),\\
&\\
&\lim_{s\rightarrow t}K(x,t,\cdot,s)=\delta_x.
\end{array}
\right.
\end{equation}

Or equivalently, if $\tau(s):=t-s$, $(y,\tau)\rightarrow K(x,t,y,\tau)$ satisfies the following forward heat equation : 

\begin{equation}
\label{scal-heat-equ-conj-bis-ker}
\left\{
\begin{array}{rl}
&\partial_{\tau}K(x,t,\cdot,\cdot)=\Delta_{g(\tau)}K(x,t,\cdot,\cdot),\\
&\\
&\partial_{\tau}g=2\Ric(g(\tau)),\\
&\\
&\lim_{\tau\rightarrow 0}K(x,t,\cdot,\tau)=\delta_x.
\end{array}
\right.
\end{equation}

\begin{prop}\label{L^1-bound-heat-kernel-fct}[$L^1$-bound]
Let $(M^n,g_0(t))_{t\geq 0}$ be a Type III solution to the Ricci flow with nonnegative scalar curvature. Then, 
\begin{eqnarray}
&&\int_MK(x,t,y,s)d\mu_{g_0(t)}(x)=1,\quad s<t,\quad y\in M,\label{L^1-est-fct}\\
 &&1\leq\int_MK(x,t,y,s)d\mu_{g_0(s)}(y)\leq\left(\frac{1+t}{1+s}\right)^c,\quad s<t,\quad x\in M,\label{L^infty-est-fct}
\end{eqnarray}
for some positive constant $c$.
\end{prop}

\begin{proof}

The proof is quite standard : let $(\Omega_j)_{j\geq 0}$ be an increasing sequence of domains of $M$ with smooth boundary exhausting the manifold $M$. Let $K_j(\cdot,\cdot,y,s)$ be the heat kernel associated to (\ref{scal-heat-equ}) on $\Omega_j$ with Dirichlet boundary condition. Then, one can prove that $K_j(\cdot,\cdot,y,s)$ is an increasing sequence converges locally uniformly to $K(\cdot,\cdot,y,s)$. By integrating by parts, one gets, for some fixed $(y,s)\in M\times [0,+\infty)$,
\begin{eqnarray*}
\partial_t\int_{\Omega_j}K_j(x,t,y,s)d\mu_{g_0(t)}(x)&=&\int_{\Omega_j}\Delta_{g_0(t)}K_j(x,t,y,s)d\mu_{g_0(t)}(x)\\
&=&\int_{\partial\Omega_j}\left<\nabla^{g_0(t)}_{x}K_j(x,t,y,s),\mathbf{n}\right>d\mu_{j,g_0(t)}(x)\\
&\leq&0,
\end{eqnarray*}
where $d\mu_{j,g_0(t)}$ is the induced measure on $\partial \Omega_j$ by $d\mu_{g_0(t)}$. Therefore, 
\begin{eqnarray*}
\int_{\Omega_j}K_j(x,t,y,s)d\mu_{g_0(t)}(x)\leq \lim_{t\rightarrow s}\int_{\Omega_j}K_j(x,t,y,s)d\mu_{g_0(t)}(x)=1,
\end{eqnarray*}
for any $t>s$ and any $j$ large enough so that $y\in \Omega_j$. By letting $j$ tending to $+\infty$, we get the first estimate. 

The second estimate is proved similarly. Remark that $K_j(x,t,\cdot,\cdot)$ satisfies (\ref{scal-heat-equ-conj}) with Dirichlet boundary condition. Therefore, 
\begin{eqnarray*}
\int_{\Omega_j}K_j(x,t,y,s)d\mu_{g_0(s)}(y)\leq \left(\frac{1+t}{1+s}\right)^c,\quad s<t,
\end{eqnarray*}
for any index $j$ large enough where $c:=\sup_{t\geq 0}(1+t)\R_{g_0(t)}$. This implies 
\begin{eqnarray*}
\int_M K(x,t,y,s)d\mu_{g_0(s)}(y)\geq \left(\frac{1+t}{1+s}\right)^c,\quad s<t.
\end{eqnarray*}
On the other hand, let $(\phi_k)_k$ be any sequence of smooth cut-off functions approximating the constant function $1$, then
\begin{eqnarray*}
\partial_t\int_M\phi_k(y)K(x,t,y,s)d\mu_{g_0(t)}(y)&=&-\int_M\Delta_{g_0(t)}\phi_k(y)K(x,t,y,s)d\mu_{g_0(t)}(y)\\
&\leq& C\sup_{\supp(\phi_k)}\arrowvert\Delta_{g_0(t)}\phi_k\arrowvert,
\end{eqnarray*}
which implies by letting $k$ tending to $+\infty$,
\begin{eqnarray*}
\partial_t\int_MK(x,t,y,s)d\mu_{g_0(t)}(y)\geq 0, 
\end{eqnarray*}
i.e. 
\begin{eqnarray*}
\int_MK(x,t,y,s)d\mu_{g_0(t)}(x)=1,\quad s<t,\quad y\in M.
\end{eqnarray*}
Similarly, one can prove the expected lower bound for $\int_MK(x,t,y,s)d\mu_{g_0(s)}(y)$.

\end{proof}

\begin{rk}\label{rk-L^1-L^infty}
In particular, proposition \ref{L^1-bound-heat-kernel-fct} tells us that the heat semigroup associated to $\Delta_{g_0(t)}+\R_{g_0(t)}$, where $(M^n,g_0(t))$ is an expanding gradient Ricci soliton with non negative scalar curvature, is $L^{1}$ contractive but not $L^{\infty}$ contractive unless it is flat.
\end{rk}

\begin{prop}[On diagonal upper bound : $L^2\rightarrow L^{\infty}$ bound]\label{diag-upper-heat-ker}
Let $(M^n,g_0(t))_{t\geq 0}$ be a non collapsed Type III solution of the Ricci flow with nonnegative Ricci curvature, i.e. 
\begin{eqnarray*}
(1+t)\arrowvert\Rm(g_0(t))\arrowvert_{g_0(t)}\leq R_0,\quad\Ric(g_0(t))\geq 0,\quad \AVR(g_0(t))\geq V_0>0,\quad  t\geq0. 
\end{eqnarray*}
Then,
\begin{eqnarray*}
0<K(x,t,y,s)\leq\frac{c(n,V_0,R_0)}{(t-s)^{\frac{n}{2}}},\quad 0\leq s<t,\quad x,y\in M.
\end{eqnarray*}
\end{prop}

\begin{proof}
It suffices to apply lemma \ref{L^1-mean-inequ} together with proposition \ref{L^1-bound-heat-kernel-fct} to the nonnegative (sub)solution $$u(x,t):=\left(\frac{1+s}{1+t}\right)^{c(R_0)}K(x,t,y,s),$$ for some fixed $(y,s)\in M\times(0,+\infty)$ with $r^2=(t-s)/2$ : 
\begin{eqnarray*}
\left(\frac{1+s}{1+t}\right)^{c(R_0)}K(x,t,y,s)&\leq&\sup_{P\left(x,t,\sqrt{\frac{t-s}{2}}\right)}\left(\frac{1+s}{1+\cdot}\right)^{c(R_0)}K(\cdot,\cdot,y,s)\\
&\leq&\left(\frac{1+s}{1+t-r^2}\right)^{c(R_0)}\frac{c(n,V_0,R_0)}{(t-s)^{\frac{n+2}{2}}}\int_{P(x,t,\sqrt{\frac{t-s}{2}})}K(x',t',y,s)d\mu_{g_0(t')}(x')dt'\\
&\leq&\left(\frac{2+2s}{2+t+s}\right)^{c(R_0)}\frac{c(n,V_0,R_0)}{(t-s)^{\frac{n}{2}}}.
\end{eqnarray*}

Hence the expected estimate.

\end{proof}
\subsection{On-diagonal bounds of the Lichnerowicz heat kernel $K_L$}
Let $(M^n,g_0(t))_{t\in[0,+\infty)}$ be an expanding gradient Ricci soliton with positive curvature operator. 
We are now in a position to estimate the heat kernel $K_L$ associated to the Lichnerowicz operator.

Recall that it satisfies, by definition, that $K_L(x,t,y,s)\in \Hom(S^2T_y^*M,S^2T_x^*M)$, $K_L$ is $C^1$ w.r.t. time and $C^2$ w.r.t. space variables, and for any fixed $(y,s)\in M\times[0,+\infty)$,

\begin{equation}
\label{Lic-heat-equ-ker}
\left\{
\begin{array}{rl}
&\partial_tK_L(\cdot,\cdot,y,s)=L_tK_L(\cdot,\cdot,y,s),\\
&\\
&\partial_tg_0=-2\Ric(g_0(t)),\\
&\\
&\lim_{t\rightarrow s}K_L(\cdot,t,y,s)=\delta_y.\\
\end{array}
\right.
\end{equation}

Now, if $(x,t)\in M\times(0,+\infty)$ is fixed, then $K_L(x,t,\cdot,\cdot)$ is the heat kernel associated to the conjugate backward heat equation : 
\begin{equation}
\label{Lic-heat-equ-conj-ker}
\left\{
\begin{array}{rl}
&\partial_sK_L(x,t,\cdot,\cdot)=-L_sK_L(x,t,\cdot,\cdot)+\R_{g_0(s)}K_L(x,t,\cdot,\cdot),\\
&\\
&\partial_sg_0=2\Ric(g_0(s)),\\
&\\
&\lim_{s\rightarrow t}K_L(x,t,\cdot,s)=\delta_x.\\
\end{array}
\right.
\end{equation}

As a corollary of theorem \ref{comp-heat-ker} and proposition \ref{diag-upper-heat-ker}, we get that the Lichnerowicz heat flow associated to the Lichnerowicz operator is ultracontractive too, i.e.

\begin{theo}[On diagonal upper bound for $K_L$ : $L^2\rightarrow L^{\infty}$ bound]\label{diag-upper-heat-ker-Lic}
Let $(M^n,g_0(t))_{t\geq 0}$ be an expanding gradient Ricci soliton with positive curvature operator.
Then,
\begin{eqnarray*}
\|K_L(x,t,y,s)\|_{\Hom(S^2T_y^*M,S^2T_x^*M)}\leq\frac{c(n,V_0,R_0)}{(t-s)^{\frac{n}{2}}},\quad 0\leq s<t,\quad x,y\in M,
\end{eqnarray*}
where $R_0:=\sup_M\R_{g_0}$, and $V_0:=\AVR(g_0)$.
\end{theo}

As a corollary of the $L^1$ estimate (\ref{L^1-est-fct}) for the operator $\Delta_{g_0(t)}+\R_{g_0(t)}$ given by proposition \ref{L^1-bound-heat-kernel-fct} and the $L^{\infty}$ estimate due to theorem \ref{Est-Hom-lin-equ}, we get that the Lichnerowicz heat flow is bounded on $L^p$ for every $p\in[1,+\infty]$ : 

\begin{theo}[$L^p\rightarrow L^p$ estimates]\label{L^p-contractivity-Lic}
Let $(M^n,g_0(t))_{t\geq 0}$ be an expanding gradient Ricci soliton with positive curvature operator.
Then, for any $p\in[1,+\infty]$,
\begin{eqnarray*}
&&\|h_t\|_{L^p(d\mu_{g_0(t)})}\leq c(p)\| h_s\|_{L^p(d\mu_{g_0(s)})},\quad t>s,\\
&&h_t(x):=K_L(x,t,\cdot,s)\ast h_s,
\end{eqnarray*}
for some positive uniform constant $c(p)\geq 1$.
\end{theo}

\begin{rk}
A priori, Theorem \ref{L^p-contractivity-Lic} does not hold for the scalar operator $\Delta_{g_0(t)}+\R_{g_0(t)}$ : see (\ref{L^infty-est-fct}) in proposition \ref{L^1-bound-heat-kernel-fct}.
\end{rk}
\begin{rk}\label{rk-contractivity-2}
The Lichnerowicz heat flow is not a priori contractive on $L^p$ unless $p=1$. In particular, the $L^2$ norm of a solution to the Lichnerowicz homogeneous equation is not necessarily decreasing.
\end{rk}

\subsection{Off-diagonal bounds for the Lichnerowicz kernel $K_L$}
Let $(M^n,g_0(t))_{t\geq 0}$ be an expanding gradient Ricci soliton with positive curvature operator.

Let $\Omega\subset M$ be a bounded open subset of $M$ with smooth boundary. Let $x_0\in \Omega$ and $(h_{x_0}(t))_{t\geq 0}$ be a one parameter family of symmetric $2$-tensors on $\Omega$ satisfying in the smooth sense,
 \begin{equation}
\label{First-step-Dir-cond}
\left\{
\begin{array}{rl}
&\partial_th_{x_0}=\Delta_{g_0(t)}h_{x_0}+\Rm(g_0(t))\ast h_{x_0},\quad\mbox{on $\Omega\times(0,+\infty)$},\\
&\\
&\partial_tg_0(t)=-2\Ric(g_0(t)),\quad\mbox{on $\Omega\times(0,+\infty)$},\\
&\\
&\nabla^{g_0(t)}\arrowvert h_{x_0}\arrowvert^2\cdot\textbf{n}\leq 0,\quad\mbox{on $\partial\Omega\times(0,+\infty)$},\\
&\\
&h_{x_0}(x,0)=0,\quad x\neq x_0,
\end{array}
\right.
\end{equation}
where $\textbf{n}$ denotes the exterior normal to $\partial \Omega$ and where $\Rm(g_0(t))\ast$ denotes any contraction of the curvature operator on $S^2T^*M$.
Consider the following integrals : 
\begin{eqnarray*}
&&I_R(t):=\int_{\Omega\setminus B_{g_0(t)}(x_0,R)}\arrowvert h_{x_0}\arrowvert^2_{g_0(t)}(x,t)d\mu_{g_0(t)}(x),\\
 &&E_D(t):=\int_{\Omega}\arrowvert h_{x_0}\arrowvert^2_{g_0(t)}(x,t)e^{2\xi_D(x,t)}d\mu_{g_0(t)}(x),\quad \xi_D(x,t):=\frac{d_{g_0(t)}^2(x_0,x)}{Dt},
\end{eqnarray*}
for $R\geq 0$ and $t>0$ such that $B_{g_0(t)}(x_0,R)\subset \Omega$.

\begin{theo}\label{Dav-Gri-Est-Ene}
Let $(M^n,g_0(t))_{t\geq 0}$ be an expanding gradient Ricci soliton with positive curvature operator.
Assume that, for some positive constant $C_0$, the following holds : 
\begin{eqnarray*}
I_{0}(t)\leq \frac{C_0}{t^{n/2}},\quad t>0.
\end{eqnarray*}

Then, for some positive constants $C_1$, $C_2$ depending on $C_0$,
\begin{eqnarray*}
&&I_R(t)\leq \frac{C_1}{t^{n/2}}e^{-\frac{2C_2}{D}\cdot\frac{R^2}{t}},\quad t>0,\quad D\geq 4,\\
&& E_D(t)\leq \frac{C(D,C_0)}{t^{n/2}},\quad t>0,\quad D>4.
\end{eqnarray*}
\end{theo}

\begin{rk}
The proof works for different behavior of $I_0(t)$ as soon as it is \textit{regular} in the sense of Grigor'yan : see \cite{Gri-Upp-Ene-Dav} for a definition.
\end{rk}

\begin{rk}\label{rk-dist-bd-back-equ}
The proof of theorem \ref{Dav-Gri-Est-Ene} neither uses the positivity of the curvature operator nor the special structure of the Ricci flow. It only uses lemma \ref{C^0-distance-est-Type-III} to get good bounds on the distance distortions. 
Therefore, it can also be applied to the backward heat equation associated to (\ref{Lic-heat-equ-conj-ker}).
\end{rk}

\begin{proof}
As in \cite{Gri-Upp-Ene-Dav}, the proof consists in three steps :
\begin{enumerate}
\item We first estimate the evolution of $I_0(t)$ as follows : 
\begin{lemma}\label{lem-evo-int}
The function $$t\in(0,+\infty)\rightarrow (1+t)^{-c(R_0)}I_0(t),$$ is non increasing, where $$R_0:=\sup_{M\times[0,+\infty)}(1+t)\arrowvert\Rm(g_0(t))\arrowvert_{g_0(t)}.$$
\end{lemma}
\begin{proof}[Proof of lemma \ref{lem-evo-int}]
We temporally simplify the notations by suppressing the dependence on the base point $x_0$.
 
By using the properties of $(h(t):=h_{x_0}(t))_{t\geq 0}$ given by (\ref{First-step-Dir-cond}),
\begin{eqnarray*}
\partial_tI_0(t)&=&\int_{\Omega}2\left<\partial_th+h\otimes \Ric(g_0(t)+\Ric(g_0(t))\otimes h-\frac{\R_{g_0(t)}}{2}h,h\right>_{g_0(t)}d\mu_{g_0(t)}\\
&=&\int_{\Omega}2\left<\Delta_{g_0(t)}h+\Rm(g_0(t))\ast h,h\right>_{g_0(t)}d\mu_{g_0(t)}\\
&\leq&-2\int_{\Omega}\arrowvert \nabla^{g_0(t)}h\arrowvert^2_{g_0(t)}d\mu_{g_0(t)}+\frac{c(n)R_0}{1+t}I_0(t)\\
&\leq&\frac{c(n)R_0}{1+t}I_0(t).
\end{eqnarray*}
Hence the result by the Gronwall inequality.
\end{proof}

Consider the following weight function 
\begin{eqnarray*}
\zeta_D(x,t):=\frac{d_{x_0,T}^2(x)}{D(t-T)},\quad x\in M,\quad T>t>0,
\end{eqnarray*}
where $d_{x_0,T}(x)$ is defined by 
\begin{equation}
\label{Dist-fct-ad-hoc}
d_{x_0,T}(x):=\left\{
\begin{array}{rl}
&R-d_{g_0(T)}(x_0,x),\quad \mbox{if $x\in B_{g_0(T)}(x_0,R)$},\\
&\\
&0,\quad \mbox{otherwise}.
\end{array}
\right.
\end{equation}
Consider the following integral quantity,
\begin{eqnarray*}
\En_D(t):=\int_{\Omega}\arrowvert h\arrowvert^2_{g_0(t)}(x,t)e^{2\zeta_D(x,t)}d\mu_{g_0(t)}(x), \quad t>0.
\end{eqnarray*}

Then we estimate $\En_D(t)$ as follows : 
\begin{lemma}\label{lem-est-Ene-evo}
The function $$t\in(0,T)\rightarrow (1+t)^{-c(R_0)}\En_D(t),$$ is non increasing for any $D\geq 4$ for some positive constant $c(R_0)$.
\end{lemma}
\begin{proof}[Proof of lemma \ref{lem-est-Ene-evo}]
We proceed as in the proof of lemma \ref{lem-evo-int} :
\begin{eqnarray*}
\partial_t\En_D(t)&=&\int_{\Omega}2\left<e^{\zeta_D}\Delta_{g_0(t)}h+\partial_t\zeta_De^{\zeta_D}h,e^{\zeta_D}h\right>_{g_0(t)}d\mu_{g_0(t)}\\
&&+\int_{\Omega}<\Rm(g_0(t))\ast e^{\zeta_D}h,e^{\zeta_D} h>_{g_0(t)}d\mu_{g_0(t)}\\
&\leq&-2\int_{\Omega}\left<\nabla^{g_0(t)}h,\nabla^{g_0(t)}\left(e^{2\zeta_D}h\right)\right>d\mu_{g_0(t)}\\
&&+2\int_{\Omega}\partial_t\zeta_De^{2\zeta_D}\arrowvert  h\arrowvert^2_{g_0(t)}d\mu_{g_0(t)}+\frac{c(R_0)}{1+t}\En_D(t)\\
&=&-2\int_{\Omega}\arrowvert \nabla^{g_0(t)}\left(e^{\zeta_D}h\right)\arrowvert^2_{g_0(t)}d\mu_{g_0(t)}\\
&&+2\int_{\Omega}\left(\arrowvert \nabla^{g_0(t)}\zeta_D\arrowvert_{g_0(t)}^2+\partial_t\zeta_D\right)e^{2\zeta_D}\arrowvert  h\arrowvert^2_{g_0(t)}d\mu_{g_0(t)}+\frac{c(R_0)}{1+t}\En_D(t).\\
\end{eqnarray*}
Now,
\begin{eqnarray*}
\partial_t\zeta_D(x,t)&=&\frac{\partial_td_{x_0,T}^2(x))}{D(t-T)}-\frac{d_{x_0,T}^2(x)}{D(t-T)^2}\\
&=&-\frac{d_{x_0,T}^2(x)}{D(t-T)^2},\\
\arrowvert\nabla^{g_0(t)}\zeta_D\arrowvert_{g_0(t)}^2(x)&=&g_0(t)^{ij}\partial_i\zeta_D\partial_j\zeta_D\\
&\leq&g_0(T)^{ij}\partial_i\zeta_D\partial_j\zeta_D\\
&\leq& \frac{4d_{x_0,T}^2(x)}{D^2(t-T)^2},
\end{eqnarray*}
where we used lemma \ref{C^0-distance-est-Type-III}.
Therefore, if $D\geq 4$,
\begin{eqnarray*}
\partial_t\zeta_D+\arrowvert\nabla^{g_0(t)}\zeta_D\arrowvert_{g_0(t)}^2\leq 0,\quad \mbox{on $M\times (0,T)$.}
\end{eqnarray*}
Hence the result by integrating the following differential inequality : 
\begin{eqnarray*}
\partial_t\En_D(t)\leq \frac{c(R_0)}{1+t}\En_D(t),\quad t\in(0,T).
\end{eqnarray*}
\end{proof}
We proceed exactly as in \cite{Gri-Upp-Ene-Dav} by estimating $I_R(t)$ from above by $I_{r}(\tau)$ for any radius $0<r<R$ and time $0<\tau<t$. By invoking lemma \ref{lem-est-Ene-evo}, we get :
\begin{eqnarray*}
\frac{1}{(1+t)^c}\int_{\Omega}\arrowvert h\arrowvert^2_{g_0(t)}(x,t)e^{2\zeta_D(x,t)}d\mu_{g_0(t)}(x)\leq \frac{1}{(1+\tau)^c}\int_{\Omega}\arrowvert h\arrowvert^2_{g_0(\tau)}(x,\tau)e^{2\zeta_D(x,\tau)}d\mu_{g_0(\tau)}(x),
\end{eqnarray*}
where we omitted the dependence of $c$ on $R_0$.
In particular,
\begin{eqnarray*}
\frac{1}{(1+t)^c}\int_{\Omega\setminus B_{g_0(T)}(x_0,R)}\arrowvert h(t)\arrowvert^2_{g_0(t)}d\mu_{g_0(t)}&\leq& \frac{1}{(1+\tau)^c}e^{-\frac{2}{D}\cdot\frac{(R-r)^2}{T-\tau}}\int_{B_{g_0(T)}(x_0,r)}\arrowvert h(\tau)\arrowvert^2_{g_0(\tau)}d\mu_{g_0(\tau)}\\
&&+ \frac{1}{(1+\tau)^c}\int_{\Omega\setminus B_{g_0(T)}(x_0,r)}\arrowvert h(\tau)\arrowvert^2_{g_0(\tau)}d\mu_{g_0(\tau)}.
\end{eqnarray*}
Now, by assumption on $I_0(t)$, one has
\begin{eqnarray*}
\int_{B_{g_0(T)}(x_0,r)}\arrowvert h(\tau)\arrowvert^2_{g_0(\tau)}d\mu_{g_0(\tau)}\leq I_0(\tau)\leq\frac{c}{\tau^{n/2}},
\end{eqnarray*}
and by lemma \ref{C^0-distance-est-Type-III}, $\Omega\setminus B_{g_0(T)}(x_0,r)\subset \Omega\setminus B_{g_0(\tau)}(x_0,r)$.
Therefore, by letting $T$ go to $t$, one gets : 
\begin{eqnarray}\label{inequ-int-step-0-induct}
\frac{1}{(1+t)^c}I_R(t)\leq \frac{1}{(1+\tau)^c}I_r(\tau)+\frac{1}{(1+\tau)^c}e^{-\frac{2}{D}\cdot\frac{(R-r)^2}{t-\tau}}
\end{eqnarray}

Consider now the following decreasing sequences of positive radii $(R_k)_{k\geq 0}$ and times $(t_k)_{k\geq 0}$ defined by :
\begin{eqnarray*}
R_k:=\left(\frac{1}{2}+\frac{1}{k+2}\right)R\quad;\quad t_k:=\frac{t}{\gamma^k}\quad;\quad k\geq 0,\quad \gamma>1.
\end{eqnarray*}
Remark that $I_{R_k}(t_k)$ goes to zero as $k$ goes to $+\infty$ since by assumption, $h(x,t)=0$ for any time $t$ and any $x\neq x_0$ and $R_k$ tends to the positive constant $R/2$.  
 Therefore, by iterating inequality (\ref{inequ-int-step-0-induct}), one has : 
 \begin{eqnarray*}
\frac{1}{(1+t)^c}I_R(t)\leq\sum_{k\geq 0}\frac{c}{t_{k+1}^{n/2}(1+t_{k+1})^c}e^{-\frac{2}{D}\cdot\frac{(R_k-R_{k+1})^2}{t_k-t_{k+1}}}.
\end{eqnarray*}
Now, a rather straightforward estimate as done in \cite{Gri-Upp-Ene-Dav} leads to the following expected estimate :
\begin{eqnarray*}
I_R(t)\leq \frac{C_1}{t^{n/2}}e^{-C_2 \frac{2}{D}\cdot\frac{R^2}{t}},
\end{eqnarray*}
where $C_1$ and $C_2$ are two positive constants depending on $\gamma$, $C_0$ and $n$ but independent of $R$ and $t$.
\item The second step is almost word for word step II in \cite{Gri-Upp-Ene-Dav}. Therefore we will be sketchy. For any positive radius $R$, 
\begin{eqnarray*}
E_D(t)&=&\int_{B_{g_0(t)}(x_0,R)}\arrowvert h(t)\arrowvert_{g_0(t)}^2e^{2\xi_D(t)}d\mu_{g_0(t)}\\
&&+\sum_{k\geq 0}\int_{B_{g_0(t)}(x_0,2^{k+1}R)\setminus B_{g_0(t)}(x_0,2^kR)}\arrowvert h(t)\arrowvert_{g_0(t)}^2e^{2\xi_D(t)}d\mu_{g_0(t)}\\
&\leq&\frac{C_0}{t^{n/2}}e^{\frac{2}{D}\cdot\frac{R^2}{t}}+\sum_{k\geq 0} e^{\frac{2}{D}\frac{4^{k+1}R^2}{t}}I_{2^kR}(t)\\
&\leq& \frac{C_0}{t^{n/2}}e^{\frac{2}{D}\cdot\frac{R^2}{t}}+\frac{C_1}{t^{n/2}}\sum_{k\geq 0} e^{\frac{2}{D}\frac{4^{k+1}R^2}{t}} e^{-\frac{2}{D_0}\frac{4^{k}R^2}{t}}\\
&\leq&  \frac{C_0}{t^{n/2}}e^{\frac{2}{D}\cdot\frac{R^2}{t}}+\frac{C_1}{t^{n/2}}\sum_{k\geq 0} e^{-\frac{2}{D}\frac{4^{k}R^2}{t}},
\end{eqnarray*}
if $D\geq 5D_0$ where $D_0$ is a fixed positive constant such that the previous step holds. Now, choose $R$ such that $R^2/Dt=\ln 2$ (this particular value is arbitrary) to get 
\begin{eqnarray*}
E_D(t)\leq \frac{C_3}{t^{n/2}}, \quad t>0,
\end{eqnarray*}
for any $D$ large enough and where $C_3$ is a positive constant independent of $t$.\\

\item The third step consists in proving the expected estimate for $E_D(t)$ for any $D>4$. It suffices to prove the estimate for $D<D_1$ where $D_1$ is a sufficiently large constant such that the previous step holds.
Adapting the proof of lemma \ref{lem-est-Ene-evo} to the weight function 
\begin{eqnarray*}
\bar{\xi}(x,t)=\frac{d_{g_0(t)}^2(x_0,x)}{4(t+T)}, \quad T>0,
\end{eqnarray*}
leads to the following inequality : 
\begin{eqnarray*}
\frac{1}{(1+t)^c}\int_{\Omega}\arrowvert h\arrowvert^2_{g_0(t)}e^{2\bar{\xi}(t)}d\mu_{g_0(t)}\leq \frac{1}{(1+\tau)^c}\int_{\Omega}\arrowvert h\arrowvert^2_{g_0(\tau)}e^{2\bar{\xi}(\tau)}d\mu_{g_0(\tau)},
\end{eqnarray*}
for any positive $0<\tau<t$. Now, we force the appearance of $E_D(t)$ (respectively of $E_{D_1}(\tau)$) on the left hand side (respectively on the right hand side) of the previous inequality. This amounts to solve the following system :
\begin{equation}
\left\{
\begin{array}{rl}
&4(t+T)=Dt,\\
&\\
&4(\tau+T)=D_1\tau,
\end{array}
\right.
\end{equation}
which has the solution 
\begin{eqnarray*}
T:=\frac{D-4}{4}t>0,\quad 0<\tau:=\frac{D-4}{D_1-4}t<t.
\end{eqnarray*}
Hence,
\begin{eqnarray*}
E_D(t)\leq\frac{C(D,C_0)}{t^{n/2}},\quad D>4,\quad t>0.
\end{eqnarray*}

\end{enumerate}

\end{proof}

\begin{rk}
The  second and third steps in the proof of theorem \ref{Dav-Gri-Est-Ene} use in a very essential way the fact that the curvature $\Rm(g_0(t))$ decays like $c/(1+t)$, the uniform boundedness of the curvature for all time alone would not  have been sufficient.
\end{rk}

We are now in a position to prove Gaussian bounds on the Lichnerowicz heat kernel $K_L$ acting on symmetric $2$-tensors : 

\begin{theo}\label{Gausian-est-scal-heat-ker}[Gaussian estimate] 
Let $(M^n,g_0(t))_{t\geq 0}$ be an expanding gradient Ricci soliton with positive curvature operator.

 Then the heat kernel associated to (\ref{scal-heat-equ}) satisfies the following Gaussian estimate : 
\begin{eqnarray*}
&&\|K_L(x,t,y,s)\|_{\Hom(S^2T_y^*M,S^2T_x^*M)}\leq \frac{c(n,V_0,R_0,D)}{(t-s)^{\frac{n}{2}}}\exp\left\{-\frac{d_{g_0(s)}^2(x,y)}{D(t-s)}\right\},\\
&& 0\leq s<t,\quad x,y\in M,\quad D>4.
\end{eqnarray*}
 where $R_0:=\sup_M\arrowvert\Rm(g_0)\arrowvert_{g_0}$ and $V_0:=\AVR(g_0)$.

\end{theo}

\begin{proof}

 Define the following integral quantities : 
 \begin{eqnarray*}
&&E^1_D(t):=\int_M\|K_L(x,t,z,s)\|_g^2e^{\frac{2d_{g_0(t)}^2(x,z)}{D(t-s)}}d\mu_{g_0(t)}(x),\quad s<t,\quad z\in M, \quad D>4,\\
&&E^2_D(s):=\int_M\| K_L(z,t,y,s)\|_g^2e^{\frac{2d_{g_0(s)}^2(z,y)}{D(t-s)}}d\mu_{g_0(s)}(y),\quad s<t,\quad z\in M,\quad D>4.\\
\end{eqnarray*}

By considering a smooth exhaustion $(\Omega_j)_j$ of $M$ and the Lichnerowicz heat kernels $(K_{L,j})_j$ solving the corresponding Dirichlet problem on $\Omega_j\times(0,+\infty)$, one can invoke theorems \ref{diag-upper-heat-ker-Lic} and \ref{L^p-contractivity-Lic} ($p=1,+\infty$) to apply theorem \ref{Dav-Gri-Est-Ene} (see remark \ref{rk-dist-bd-back-equ}) to $K_{L,j}$ for each index $j$ in order to get the following uniform integral estimates :
 \begin{eqnarray*}
&&E^1_D(t)\leq \frac{C(n,V_0,R_0,D)}{(t-s)^{n/2}},\quad s<t,\quad D>4,\\
&&E^2_D(s)\leq  \frac{C(n,V_0,R_0,D)}{(t-s)^{n/2}},\quad s<t,\quad D>4.
\end{eqnarray*}

Now, by the semi-group property, the triangular inequality and lemma \ref{C^0-distance-est-Type-III},
\begin{eqnarray*}
&&\|K_L(x,t,y,s)\|_g\\
&=&\left\|\int_MK_L(x,t,z,\tau)\circ K_L(z,\tau,y,s)d\mu_{g_0(\tau)}(z)\right\|_g\\
&\leq&\int_M\|K_L(x,t,z,\tau)\|_{g}e^{\frac{d_{g_0(\tau)}^2(x,z)}{D(t-\tau)}}\|K_L(z,\tau,y,s)\|_{g}e^{\frac{d_{g_0(\tau)}^2(z,y)}{D(\tau-s)}}d\mu_{g_0(\tau)}(z)e^{-\frac{d_{g_0(\tau)}^2(x,y)}{D(t-s)}}\\
&\leq&C(R_0)\int_M\|K_L(x,t,z,\tau)\|_ge^{\frac{d_{g_0(t)}^2(x,z)}{D(t-\tau)}}\|K_L(z,\tau,y,s)\|_ge^{\frac{d_{g_0(s)}^2(z,y)}{D(\tau-s)}}d\mu_{g_0(\tau)}(z)e^{-\frac{d_{g_0(\tau)}^2(x,y)}{D(t-s)}}
\end{eqnarray*}
where $\|\cdot\|_g$ stands for the norm on $\Hom(S^2T_{\cdot}^*M,S^2T_{\cdot}^*M)$ and $$s<\tau:=\frac{t+s}{2}<t.$$
Therefore, by the Cauchy-Schwarz inequality and lemma \ref{C^0-distance-est-Type-III}, we get the following "universal" inequality :
\begin{eqnarray*}
\|K_L(x,t,y,s)\|_g\leq C(R_0) \sqrt{E^1_D(t)E^2_D(\tau)}\exp\left\{-\frac{d_{g_0(\tau)}^2(x,y)}{D(t-s)}\right\},
\end{eqnarray*}
which implies the expected Gaussian estimate : 
\begin{eqnarray*}
\|K_L(x,t,y,s)\|_g&\leq&\frac{c(n,V_0,R_0,D)}{(t-s)^{\frac{n}{2}}}\exp\left\{-\frac{d_{g_0(\tau)}^2(x,y)}{D(t-s)}\right\}\\
&\leq&\frac{c(n,V_0,R_0,D)}{(t-s)^{\frac{n}{2}}}\exp\left\{-\frac{d_{g_0(s)}^2(x,y)}{D(t-s)}\right\},\quad 0\leq s<t,\quad x,y\in M,
\end{eqnarray*}
where we used lemma \ref{C^0-distance-est-Type-III} to estimate from below the distance $d_{g_0(\tau)}(x,y)$ in terms of 
the distance $d_{g_0(s)}(x,y)$.

\end{proof}

\subsection{Miscellaneous results on the bottom of the spectrum of the (weighted) Lichnerowicz operator}
\subsubsection{Links with the weighted Lichnerowicz heat kernel}\label{subsec-lic-hea-ker}
We first investigate the bottom of the spectrum of the weighted Lichnerowicz operator.

\begin{lemma}\label{bottom-spec-weighted-op}
Let $(M^n,g,\nabla^gf)$, $n\geq 2$, be an expanding gradient Ricci soliton with nonnegative sectional curvature. Then the bottom of the $L^2(e^fd\mu(g))$-spectrum of the weighted Lichnerowicz operator is bounded from below by $n/2$, i.e.
\begin{eqnarray*}
\lambda_1(-\Delta_f-2\Rm(g)\ast)\geq \frac{n}{2}.
\end{eqnarray*}
If $n=2$ then $\lambda_1(-\Delta_f-2\Rm(g)\ast)=1=n/2.$ If $n\geq 3$ and $(M^n,g)$ has positive sectional curvature then $\lambda_1(-\Delta_f-2\Rm(g)\ast)>n/2$.
\end{lemma}

\begin{rk}
In \cite{Der-Sta-Sge}, we proved that $\lambda_1(-\Delta_f-2\Rm(g)\ast)>0$ if $(M^n,g,\nabla^gf)$ is an expanding gradient Ricci soliton with positive curvature operator. The argument was proved by contradiction using a maximum principle type argument.
\end{rk}

\begin{proof}
By the variational characterization of $\lambda_1(-\Delta_f-2\Rm(g)\ast)$, it suffices to prove that 
\begin{eqnarray*}
\inf_{h\in  L^2(S^2T^*M,d\mu_f(g))\setminus\{0\}}\frac{\int_M\left<-\Delta_fh-2\Rm(g)\ast h,h\right>d\mu_f(g)}{\int_M\arrowvert h\arrowvert^2_gd\mu_f(g)}\geq \frac{n}{2}.
\end{eqnarray*}
By \cite{Der-Sta-Sge}, 
\begin{eqnarray*}
&&\int_M\left<-\Delta_f\phi,\phi\right>d\mu_f(g)\geq \int_M(\Delta_g f)\phi^2d\mu_f(g),\quad\phi\in C^{\infty}_0(M),\\
&&\Delta_gf=\R_g+\frac{n}{2}.
\end{eqnarray*}
Therefore, if $h$ is a symmetric $2$-tensor with compact support, we get by the Kato inequality,
\begin{eqnarray*}
\int_M\left<-(\Delta_f+2\Rm(g)\ast)h,h\right>d\mu_f(g)\geq \int_M \left<\left(\R_g-2\Rm(g)\ast\right)h,h\right>d\mu_f(g)+\frac{n}{2}\int_M\arrowvert h\arrowvert_g^2d\mu_f(g).
\end{eqnarray*}
Hence the result if we prove the pointwise inequality : $\R_g\arrowvert h\arrowvert^2-2\left<\Rm(g)\ast h,h\right>\geq 0$, for any symmetric $2$-tensors. Indeed, let $h\in S^2T^*M$ and $x\in M$. Let $(e_i)_{1\leq i\leq n}$ be an orthonormal base of $(T_xM,g(x))$ diagonalizing $h$, i.e. $h(e_i,e_j)=\lambda_i\delta_{ij}$, with $\lambda_i\in\mathbb{R}$. Then,
\begin{eqnarray*}\label{inequ-point-curv-term}
\R_g\arrowvert h\arrowvert^2-2\left<\Rm(g)\ast h,h\right>&=&\left(2\sum_{i<j}K_{ij}\right)\left(\sum_{k=1}\lambda_k^2\right)-4\sum_{i<j}K_{ij}\lambda_i\lambda_j\\
&=&2\sum_{i<j}K_{ij}\left((\lambda_i-\lambda_j)^2+\sum_{k\neq i,j}\lambda_k^2\right)\\
&\geq &0
\end{eqnarray*}
 where $K_{ij}:=\Rm(g)(e_i,e_j,e_j,e_i)$ denotes the sectional curvature with respect to the metric $g$ of the $2$-plane spanned by $(e_i,e_j)$.
 Now, if $n=2$, then
 \begin{eqnarray*}
(\Delta_f+2\Rm(g)\ast )h=\Delta_f h+\R_gh, \quad h\in S^2T^*M.
\end{eqnarray*}
Hence, if $h:=e^{-f}g$, then $(\Delta_f+2\Rm(g)\ast )h=-h$ and $e^{-f}\in L^2(d\mu_f(g))$.

If $n\geq 3$ and $(M^n,g)$ has positive sectional curvature, then, as shown in \cite{Der-Sta-Sge}, there is some eigentensor $h\in L^2(S^2T^*M,d\mu_f(g))\setminus\{0\}$ associated to the bottom eigenvalue of the spectrum of the weighted Lichnerowicz operator. The previous pointwise inequalities imply easily $\lambda_1(-\Delta_f-2\Rm(g)\ast)>n/2$.

\end{proof}

If $(M^n,g,\nabla^g f)$ is an expanding gradient Ricci soliton, the heat kernel $K_{L,f}$ associated to the weighted Lichnerowicz operator $\Delta_f+2\Rm(g)\ast$ is linked to the (classical) Lichnerowicz heat kernel $K_L$ studied in the previous sections by the following formula :
\begin{eqnarray*}
K_L(x,t,y,0)=(1+t)\phi_{t,x}^*K_{L,f}(\ln(1+t),x,y)e^{f(y)},\quad x,y\in M,\quad t>0,
\end{eqnarray*}
 where $\phi_{t,x}^*$ denotes the pull-back with respect to the $x$ variable by the one parameter family of diffeomorphisms $(\phi_t)_{t>-1}$ generated by $-\nabla^gf/(1+t)$. 
 
 As suggested in the introduction, the kernel $K_{L,f}$ is the Mehler kernel in the case $(M^n,g,\nabla^g f)$ is flat. Therefore, theorem \ref{Gausian-est-scal-heat-ker} enables to prove Mehler bounds (instead of Gaussian) for the heat kernel $K_{L,f}$ in case the expander has positive curvature operator. 

\subsubsection{Hardy inequalities on Ricci expanders}

Regarding remark \ref{rk-contractivity-2}, one can ask when the Lichnerowicz heat flow is contractive on $L^2$. It turns out that it is equivalent to prove that the operator $-\Delta_{g}-2\Rm(g)\ast+\R_{g}/2$ is non negative. We provide a yes answer with the help of a suitable Hardy inequality.

\begin{prop}
Let $(M^n,g,\nabla^{g}f)$ be a normalized expanding gradient Ricci soliton with nonnegative Ricci curvature and positive scalar curvature (i.e. non flat). Then the following Hardy inequality holds :
\begin{eqnarray*}
\left(\frac{n-2}{2}\right)^2\int_{M^n}\left(\frac{\phi}{2\sqrt{f+\mu(g)})}\right)^2d\mu(g)\leq\int_{M^n}\arrowvert\nabla^{g} \phi\arrowvert^2d\mu(g),
\end{eqnarray*}
for any $\phi\in C_0^{\infty}(M^n).$
\end{prop}

\begin{proof}
By the soliton identities of lemma \ref{id-EGS}, observe that if $F:=f+\mu(g)> 0$,
\begin{eqnarray*}
\Delta_{g}(2F^{1/2})&=&\div_{g}\left(\frac{\nabla^{g} f}{\sqrt{F}}\right)\\
&=&\frac{2\Delta_{g}f-\arrowvert\nabla^{g}f\arrowvert^2F^{-1}}{2\sqrt{F}}\\
&=&\frac{2R_{g}+n-1+\R_{g}F^{-1}}{2\sqrt{F}}\\
&\geq&\frac{n-1}{2\sqrt{F}}.
\end{eqnarray*}

Define $\rho:=2\sqrt{F}.$ Then,
\begin{eqnarray*}
\Delta_{g}\rho^{-(\frac{n-2}{2})}&=&-\frac{n-2}{2}\rho^{-\frac{n-2}{2}-1}\Delta_{g}\rho+\frac{n-2}{2}\left(\frac{n-2}{2}+1\right)\arrowvert\nabla^{g}\rho\arrowvert^2\rho^{-\frac{n-2}{2}-2}\\
&\leq&\left(-\frac{(n-1)(n-2)}{2}+\frac{n-2}{2}\left(\frac{n-2}{2}+1\right)\right)\rho^{-\frac{n-2}{2}-2}\\
&\leq&-\left(\frac{n-2}{2}\right)^2\cdot\frac{1}{\rho^{2}}\cdot\rho^{-\frac{n-2}{2}}.
\end{eqnarray*}
Therefore, if $\psi:=\rho^{-\frac{n-2}{2}}$, we have shown that
\begin{eqnarray*}
\Delta_{g}\psi\leq-\left(\frac{n-2}{2}\right)^2\frac{\psi}{\rho^2}.
\end{eqnarray*}
Now, by an argument due to Carron \cite{Car-Ine-Har}, let $\phi\in C_0^{\infty}(M)$, and note that
\begin{eqnarray*}
\int_{M^n}\arrowvert\nabla^{g}(\psi\phi)\arrowvert^2d\mu(g)&=&-\int_{M^n}\phi\psi\Delta_{g}(\psi\phi)d\mu(g)\\
&=&\int_{M^n}\psi^2\arrowvert\nabla \phi\arrowvert^2d\mu(g)-\int_{M^n}\phi^2\psi\Delta_{g}\psi d\mu(g)\\
&\geq&\left( \frac{n-2}{2}\right)^2\int_{M^n}\left(\frac{\phi\psi}{\rho}\right)^2d\mu(g).
\end{eqnarray*}

\end{proof}

\begin{coro}
Let $(M^n,g,\nabla^{g}f)$ be a normalized expanding gradient Ricci soliton with nonnegative Ricci curvature. 
Then there exists a positive constant $c(n)$ such that if 
\begin{eqnarray*}
\sup_M(f+\mu(g))\arrowvert\Rm(g)\arrowvert\leq c(n),
\end{eqnarray*}
 i.e. if the aperture of the asymptotic cone is large enough, then
\begin{eqnarray*}
\lambda_1(-\Delta_{g}-2\Rm(g)\ast+\R_{g}/2)\geq 0.
\end{eqnarray*}
\end{coro}


\subsection{Bounds on the covariant derivatives of the heat kernel}
Once one has pointwise estimates on solutions of heat equations type, pointwise estimates on the covariant derivatives follows automatically as soon as the potential is controlled pointwisely. 
\begin{prop}[Covariant derivatives bounds]\label{cov-der-bound}
Let $(M^n,g_0(t))_{t\geq 0}$ be a Type III solution of the Ricci flow with nonnegative Ricci curvature, i.e. 
\begin{eqnarray*}
(1+t)\arrowvert\Rm(g_0(t))\arrowvert_{g_0(t)}\leq R_0,\quad \Ric(g_0(t))\geq 0,\quad  t\geq0. 
\end{eqnarray*}
Let $h$ be a tensor solution to the following system : 
\begin{eqnarray}
\label{sys-heat-equ-cov-der}
\partial_th=\Delta_{g_0(t)}h+\Rm(g_0(t))\ast h,
\end{eqnarray}
on $M^n\times(0,+\infty)$. Then, for any integer $k\geq 0$,
\begin{eqnarray*}
\sup_{P(x,t,\theta r)\setminus P(x,t,\theta r/2)}s^{k}\arrowvert \nabla^{g_0(s),k}h\arrowvert^2_{g_0(s)}\leq c(n,k,R_0,\theta)\left(1+\frac{t}{r^2}+\frac{t}{t-r^2}\right)\sup_{P(x,t,r)}\arrowvert h\arrowvert^2_{g_0(s)},
\end{eqnarray*}
where $(x,t)\in M\times(0,+\infty)$, and $r$ is such that $t>r^2>0$.
\end{prop}

\begin{proof}
The proof is in the spirit of Shi's estimates \cite{Shi-Def}. Again, we have to pay attention to the time dependence of the constants. We prove the case $k=1$, the other cases can be proved similarly. We compute the evolution of the first covariant derivative of $h$ as in the proof of theorem \ref{Est-Hom-lin-equ}: 
\begin{eqnarray*}
\partial_t\arrowvert\nabla^{g_0(t)}h\arrowvert^2_{g_0(t)}&\leq&\Delta_{g_0(t)}\arrowvert\nabla^{g_0(t)}h\arrowvert^2_{g_0(t)}-2\arrowvert\nabla^{g_0(t),2}h\arrowvert_{g_0(t)}\\
&&+c(n)\arrowvert\nabla^{g_0(t)}\Rm(g_0(t))\arrowvert_{g_0(t)}\arrowvert h\arrowvert_{g_0(t)}\arrowvert\nabla^{g_0(t)}h\arrowvert_{g_0(t)}\\
&&+c(n)\arrowvert\Rm(g_0(t))\arrowvert_{g_0(t)}\arrowvert\nabla^{g_0(t)}h\arrowvert_{g_0(t)}^2,\\
\partial_t\left(t\arrowvert\nabla^{g_0(t)}h\arrowvert^2_{g_0(t)}\right)&\leq&\Delta_{g_0(t)}\left(t\arrowvert\nabla^{g_0(t)}h\arrowvert^2_{g_0(t)}\right)-2t\arrowvert\nabla^{g_0(t),2}h\arrowvert_{g_0(t)}\\
&&+\frac{c(n,R_0)}{1+t}\left(\arrowvert h\arrowvert^2_{g_0(t)}+t\arrowvert\nabla^{g_0(t)}h\arrowvert^2_{g_0(t)}\right)\\
&&+\arrowvert\nabla^{g_0(t)}h\arrowvert^2_{g_0(t)},
\end{eqnarray*}
where we used $(1+t)^{3/2}\arrowvert\nabla^{g_0(t)}\Rm(g_0(t))\arrowvert_{g_0(t)}\leq c(n,R_0)$ for $t\geq 0$.
Define $u_0:=\arrowvert h\arrowvert^2_{g_0(t)}$, $u_1:=t\arrowvert\nabla^{g_0(t)}h\arrowvert_{g_0(t)}^2$, for $t\geq 0$.
Then,
\begin{eqnarray*}
&&\partial_tu_0\leq \Delta_{g_0(t)}u_0-2\frac{u_1}{t}+\frac{c(n,R_0)}{1+t}u_0\\
&&\partial_tu_1\leq \Delta_{g_0(t)}u_1-2t\arrowvert\nabla^{g_0(t),2}h\arrowvert_{g_0(t)}^2+\frac{u_1}{t}+\frac{c(n,R_0)}{1+t}(u_0+u_1).
\end{eqnarray*}
Consider now the function $u_1(u_0+a)$ where $a$ is a positive constant to be defined later : 
\begin{eqnarray*}
&&\partial_t(u_1(u_0+a))\leq \Delta_{g_0(t)}(u_1(u_0+a))-2t\arrowvert\nabla^{g_0(t),2}h\arrowvert_{g_0(t)}^2(u_0+a)\\
&&+2\arrowvert\nabla^{g_0(t)}u_0\arrowvert_{g_0(t)}\arrowvert\nabla^{g_0(t)}u_1\arrowvert_{g_0(t)}+\frac{u_1(u_0+a)}{t}-2\frac{u_1^2}{t}+\frac{c(n,R_0)}{1+t}(u_1(u_0+a)+a^2)\\
&&\leq\Delta_{g_0(t)}(u_1(u_0+a))-2t\arrowvert\nabla^{g_0(t),2}h\arrowvert_{g_0(t)}^2(u_0+a)\\
&&+8u_0^{1/2}u_1\arrowvert\nabla^{g_0(t),2}h\arrowvert_{g_0(t)}+\frac{u_1(u_0+a)}{t}-2\frac{u_1^2}{t}+\frac{c(n,R_0)}{1+t}(u_1(u_0+a)+a^2)\\
&&\leq\Delta_{g_0(t)}(u_1(u_0+a))+\frac{u_1(u_0+a)}{t}-\frac{u_1^2}{t}+\frac{c(n,R_0)}{1+t}(u_1(u_0+a)+a^2)\\
&&\leq\Delta_{g_0(t)}(u_1(u_0+a))-\frac{((u_1(u_0+a))^2}{a^2t}+\frac{c(n,R_0)}{t}(u_1(u_0+a)+a^2),
\end{eqnarray*}
if $a= c\sup_{P(x_0,t_0,r_0)}u_0$, $x_0\in M$, $t_0>0$, $r_0^2>0$ being fixed such that $t_0-r_0^2>0$, where $c$ is a positive "universal" constant. Consider the following cut-off function $\psi$ as in the proof of theorem \ref{Est-Hom-lin-equ} : choose two smooth functions $\phi:\mathbb{R}_+\rightarrow[0,1]$ and $\eta:\mathbb{R}_+\rightarrow[0,1]$ such that
\begin{eqnarray*}
&& \supp(\phi)\subset [0,r],\quad\phi\equiv 1\quad\mbox{in $[0,\theta r]$},\quad \phi\equiv 0\quad\mbox{in $[r,+\infty)$},\\
 &&-\frac{c}{r(1-\theta)}\leq \phi'\leq 0,\quad\frac{(\phi')^2}{\phi}\leq \frac{c}{((1-\theta)r)^2},\quad \phi''\geq -\frac{c}{((1-\theta)r)^2},\\
&& \supp(\eta)\subset [t-r^2,t],\quad\eta\equiv 1\quad\mbox{in $[t-(\theta r)^2,t-((\theta r)/2)^2]$},\\
&& \eta\equiv 0\quad\mbox{in $(0,t-r^2]\cup [t,+\infty)$},\quad  \arrowvert\eta'\arrowvert\leq \frac{c(\theta)}{r^2}.
\end{eqnarray*}
Define $\psi(y,s):=\phi(d_{g_0(s)}(x,y))\eta(s)$, for $(y,s)\in M\times(0,+\infty)$. Then, the function $\psi u_1(u_0+a)$ attains its maximum on $M\times(0,+\infty)$ at some space-time point $(x_1,t_1)$.  Now, at such a space-time point, $\psi u_1(u_0+a)$ satisfies : 
\begin{eqnarray*}
0&=&\psi\partial_t(\psi u_1(u_0+a))\\
&\leq& \psi\Delta_{g_0(t)}(\psi u_1(u_0+a))-\frac{((\psi u_1(u_0+a))^2}{a^2t}+\frac{c(n,R_0)}{t_1}(\psi u_1(u_0+a)+a^2)\\
&&+(\partial_t-\Delta_{g_0(t)})\psi\left(\psi u_1(u_0+a)\right)-2\psi<\nabla^{g_0(t)}\psi,\nabla^{g_0(t)}(u_1(u_0+a))>_{g_0(t)}\\
&\leq&-\frac{((\psi u_1(u_0+a))^2}{a^2t_1}+\frac{c(n,R_0)}{t_1}(\psi u_1(u_0+a)+a^2)+2\frac{\arrowvert\nabla^{g_0(t_1)}\psi\arrowvert_{g_0(t_1)}^2}{\psi}(\psi u_1(u_0+a))\\
&&+(\partial_t-\Delta_{g_0(t_1)})\psi\left(\psi u_1(u_0+a)\right).
\end{eqnarray*}
Let us estimate $(\partial_t-\Delta_{g_0(t_1)})\psi$ as follows : 
\begin{eqnarray*}
\partial_t\psi&=&\eta'\phi+\eta\phi'\partial_td_{g_0(\cdot)}(x,\cdot)\leq c(\theta)\left(\frac{1}{r^2}+\frac{1}{\sqrt{t-r^2}r}\right)\\
\Delta_{g_0(s)}\psi&=&\eta\phi'\Delta_{g_0(s)}(d_{g_0(s)}(x,\cdot))+\eta\phi''\arrowvert\nabla^{g_0(s)}d_{g_0(\cdot)}\arrowvert^2_{g_0(s)},\\
&\geq&-\frac{c(n)}{d_{g_0(s)}(x,\cdot)}\cdot\frac{1}{(1-\theta)r}-\frac{c}{((1-\theta)r)^2}\\
&\geq&-c(n)\left(\frac{1}{\theta (1-\theta)r^2}+\frac{1}{((1-\theta)r)^2}\right)\geq -\frac{c(n,\theta)}{r^2},
\end{eqnarray*}
where we used the estimate 
\begin{eqnarray*}
\Delta_{g_0(s)}d_{g_0(s)}(x,\cdot)\leq\frac{n-1}{d_{g_0(s)}(x,\cdot)},
\end{eqnarray*}
that holds outside the cutlocus of $x$ since $\Ric(g_0(s))\geq 0$ together with lemma \ref{C^0-distance-est-Type-III}.
Therefore, at $(x_1,t_1)$, the maximum of $\psi u_1(u_0+a)$ denoted by $S$ satisfies : 
\begin{eqnarray*}
S^2\leq c(n,R_0,\theta)a^2\left(1+\frac{t}{r^2}+\frac{t}{t-r^2}\right)S+c(n,R_0)a^4,
\end{eqnarray*}
which implies 
\begin{eqnarray*}
a\sup_{P(x,t,\theta r)\setminus P(x,t,\theta r/2)}u_1\leq S=\sup_{M\times (0,+\infty)}\psi u_1(u_0+a)\leq c(n,R_0,\theta)a^2\left(1+\frac{t}{r^2}+\frac{t}{t-r^2}\right),
\end{eqnarray*}
that is,
\begin{eqnarray*}
\sup_{P(x,t,\theta r)\setminus P(x,t,\theta r/2)}s\arrowvert\nabla^{g_0(s)}h\arrowvert^2_{g_0(s)}\leq c(n,R_0,\theta)\left(1+\frac{t}{r^2}+\frac{t}{t-r^2}\right)\sup_{P(x,t,r)}\arrowvert h\arrowvert^2_{g_0(s)},
\end{eqnarray*}
which is the desired estimate.
\end{proof}
\section{Estimates for the inhomogeneous equation}\label{sec-est-inhom-equ}
We finish the proof of theorem \ref{main-theo} by proving the following theorem :
\begin{theo}
Let $R$ be in $Y$. Then any solution to $\left(\partial_t-L_t\right)h=R$ with $h(0)=h_0\in L^{\infty}(M,g_0)$  of the form $h=K_L\ast R$, i.e.
\begin{eqnarray*}
h(x,t)=\int_0^t\int_M<K_L(x,t,y,s),R(y,s)>_{g_0}d\mu_{g_0(s)}ds,
\end{eqnarray*}

is in $X$ and satisfies the following estimate : 
\begin{eqnarray*}
\| h\|_X\leq c(\| h_0\|_{L^{\infty}(M,g_0)}+\| R\|_Y),
\end{eqnarray*}
for some positive uniform constant $c$.
\end{theo}

\begin{proof}

By theorem \ref{Est-Hom-lin-equ}, it suffices to consider the case $h_0=0$. The proof is very similar to what is done in [Lemma $4.2$, \cite{Koc-Lam-Rou}]. According to the high number of estimates that need to be proved, we separate each into three lemmata.

\begin{lemma}[$L^{\infty}$ estimate]\label{L-infty-non-hom}
Let $R$ be in $Y$. Then any solution of the form $h:=K_L\ast R$ satisfies the following estimate : 
\begin{eqnarray*}
\sup_{t\geq 0}\| h\|_{L^{\infty}(M,g_0(t))}\leq c\| R\|_{Y},
\end{eqnarray*}
for some positive uniform constant $c$.
\end{lemma}

\begin{proof}[Proof of lemma \ref{L-infty-non-hom}]

\begin{itemize}
\renewcommand{\labelitemi}{$\bullet$}

\item Let us estimate the $R_0$ term.

We cut the integral into two pieces with respect to $Q(x,R):=P(x,R)\setminus P(x,R/\sqrt{2})$, for $(x,R)\in M\times \mathbb{R}_+^*$ :
\begin{eqnarray*}
\arrowvert K_L\ast R_0\arrowvert(x,t)&\leq& \int_{Q(x,\sqrt{t})}\arrowvert K_L(x,t,y,s)\arrowvert_{g_0}\arrowvert R_0\arrowvert(y,s)d\mu_{g_0(s)}ds\\
&&+ \int_{M\setminus Q(x,\sqrt{t})}\arrowvert K_L(x,t,y,s)\arrowvert_{g_0}\arrowvert R_0\arrowvert(y,s)d\mu_{g_0(s)}ds=:I_{Q(x,\sqrt{t})}+I_{Q^c(x,\sqrt{t})}.
\end{eqnarray*}
Now, on one hand, by theorem \ref{Gausian-est-scal-heat-ker} together with H\"older's inequality,
\begin{eqnarray*}
I_{Q(x,\sqrt{t})}&\leq& \|K_L(x,t,\cdot,\cdot)\|_{L^{\frac{n+4}{n+2}}(Q(x,\sqrt{t}))}\|R_0\|_{L^{\frac{n+4}{2}}(Q(x,\sqrt{t}))}\\
&\leq& C(n,g_0)t^{\frac{2}{n+4}}\|R_0\|_{L^{\frac{n+4}{2}}(Q(x,\sqrt{t}))}\\
&\leq& C(n,g_0)\nor tR_0\|_{L^{\frac{n+4}{2}}(Q(x,\sqrt{t}))}\leq C(n,g_0) \|R_0\|_{Y_0},
\end{eqnarray*}
where $C(n,g_0)$ is a positive constant depending on $n$ and $g_0(=g_0(0))$ only.

On the other hand, notice that $$Q^c(x,\sqrt{t})\subset ((0,t/2]\times M)\cup (t/2,t)\times M\setminus B_{g_0(t/2)}(x,\sqrt{t}).$$ Therefore, by covering $M\setminus B_{g_0(t)}(x,\sqrt{t})$ by balls of radius $\sqrt{t}$ as follows : $$ M\setminus B_{g_0(t)}(x,\sqrt{t})=\cup_{i\geq 1}\cup_{j=1}^{j(i)}B_{g_0(t)}(x_j,\sqrt{t}),$$ where $d_{g_0(t)}(x_j,x)=i\sqrt{t}$ when $j\in\{1,...,j(i)\}$. One has $j(i)\leq C(n,g_0)i^n$ by volume comparison. In particular, by using theorem \ref{Gausian-est-scal-heat-ker} (fixing the constant $D>4$ once and for all), if $x_0:=x$,
\begin{eqnarray*}
I_{Q^c(x,\sqrt{t})}&\leq&C(n,g_0) \sum_{i\geq 0}\sum_{j=0}^{j(i)}e^{-C(n,g_0)i^2}t^{-n/2}\int_{0}^t\int_{B_{g_0(t)}(x_j,\sqrt{t})}\arrowvert R_0\arrowvert(y,s)d\mu_{g_0(s)}ds\\
&\leq& C(n,g_0)\sup_{i\geq 0}\sup_{j\in[0,j(i)]}t\nor R_0\|_{L^1(P(x_j,\sqrt{t}))}\\
&\leq& C(n,g_0)\|R_0\|_{Y_0},
\end{eqnarray*}
by using lemma \ref{C^0-distance-est-Type-III} repeatedly. \\

\item Finally, the $R_{1}$ term can be handled in the same way after integrating by parts by using proposition \ref{cov-der-bound}.
\end{itemize}
\end{proof}

\begin{lemma}[$L^{2}$ estimate]\label{L-2-non-hom}
Let $R$ be in $Y$. Then any solution of the form $h:=K_L\ast R$ satisfies the following estimate : 
\begin{eqnarray*}
\sup_{(x,t)\in M\times \mathbb{R}_+^*}\sqrt{t}\nor  \nabla h\|_{L^2(P(x,\sqrt{t}))}\leq c\| R\|_{Y},
\end{eqnarray*}
for some positive uniform constant $c$.
\end{lemma}
\begin{proof}[Proof of lemma \ref{L-2-non-hom}]
The proof goes along the same lines of the proof of the corresponding estimate for the homogeneous equation : see the proof of theorem \ref{Est-Hom-lin-equ}.

Since $(h(t))_{t\geq 0}$ satisfies $\partial_th=L_th+R$, one has,
\begin{eqnarray*}
\partial_t\arrowvert h\arrowvert_{g_0(t)}^2\leq \Delta_{g_0(t)}\arrowvert h\arrowvert^2_{g_0(t)}-2\arrowvert\nabla^{g_0(t)}h\arrowvert_{g_0(t)}^2+c(n)\arrowvert\Rm(g_0(t))\arrowvert_{g_0(t)}\arrowvert h\arrowvert_{g_0(t)}^2+2< R,h>_{g_0(t)}.
\end{eqnarray*}
Multiplying this inequality by any smooth cut-off function $\psi : M\times\mathbb{R}_+\rightarrow \mathbb{R}_+$ gives, after integrating by parts (in space) and using Young's inequality : 
\begin{eqnarray*}
&&\partial_t\int_M\psi^2\arrowvert h\arrowvert^2_{g_0(t)}d\mu_{g_0(t)}+\int_M\psi^2\arrowvert \nabla^{g_0(t)}h\arrowvert_{g_0(t)}^2d\mu_{g_0(t)}\leq\\
&& c(n)\int_M\left(\arrowvert \nabla^{g_0(t)}\psi\arrowvert_{g_0(t)}^2+\partial_t\psi^2+\psi^2\arrowvert\Rm(g_0(t))\arrowvert_{g_0(t)}\right)\arrowvert h\arrowvert^2_{g_0(t)}d\mu_{g_0(t)}\\
&&+c(n)\int_M\psi^2\arrowvert R_1\arrowvert_{g_0(t)}^2+\arrowvert R_0\arrowvert_{g_0(t)}\psi^2\arrowvert h\arrowvert_{g_0(t)}d\mu_{g_0(t)}.
\end{eqnarray*}
Now, let $(x,t)\in M\times\mathbb{R}_+^*$ and consider the following cutoff function : $\psi_{x,t}(y,s):=\psi(d_{g_0(s)}(x,y)/\sqrt{t}),$ where $\psi:\mathbb{R}_+\rightarrow \mathbb{R}_+ $ is a smooth function such that $\psi_{|[0,1]}\equiv 1$, $\psi_{|[2,+\infty)}\equiv 0$ and $\sup_{\mathbb{R}_+}\arrowvert\psi'\arrowvert\leq c$. $\psi_{x,t}$ is a Lipschitz function satisfying :
\begin{eqnarray*}
&&\arrowvert\nabla^{g_0(s)}\psi_{x,t}\arrowvert_{g_0(s)}\leq \frac{c}{\sqrt{t}},\quad0\leq\partial_s\psi_{x,t}\leq \frac{c}{\sqrt{t}\sqrt{s}},
\end{eqnarray*}
almost everywhere, by lemma \ref{C^0-distance-est-Type-III}. Hence, by integrating in time together with lemma \ref{L-infty-non-hom} : 
\begin{eqnarray*}
&&\int_0^t\int_{B_{g_0(s)}(x,\sqrt{t})}\arrowvert\nabla^{g_0(s)}h\arrowvert^2_{g_0(s)}d\mu_{g_0(s)}ds\leq\\
&& c(n)\left(t^{\frac{n}{2}}+\int_0^{t}\int_{B_{g_0(s)}(x,2\sqrt{t})}\arrowvert\Rm(g_0(s))\arrowvert_{g_0(s)}d\mu_{g_0(s)}ds\right)\| h\|^2_{L^{\infty}(P(x,2\sqrt{t}))}\\
&&+c(n)t^{\frac{n}{2}}\| R\|^2_Y\leq c(n,g_0)t^{\frac{n}{2}}\| R\|^2_Y,
\end{eqnarray*}
since the curvature of $g_0$ has quadratic curvature decay : see the proof of theorem \ref{Est-Hom-lin-equ}.

\end{proof}

\begin{lemma}[$L^{n+4}$ estimate]\label{L-(n+4)-non-hom}
Let $R$ be in $Y$. Then any solution of the form $h:=K_L\ast R$ satisfies the following estimate : 
\begin{eqnarray*}
\sup_{(x,t)\in M\times \mathbb{R}_+^*}\sqrt{t}\nor  \nabla h\|_{L^{n+4}(P(x,\sqrt{t})\setminus P(x,\sqrt{t}/2))}\leq c\| R\|_{Y},
\end{eqnarray*}
for some positive uniform constant $c$.
\end{lemma}
\begin{proof}[Proof of lemma \ref{L-(n+4)-non-hom}]
As noticed in \cite{Koc-Lam-Rou}, the method used in the proof of lemma \ref{L-infty-non-hom} to estimate $I_{Q^c(x_0,\sqrt{t_0})}$ for $(x_0,t_0)\in M\times\mathbb{R}_+^*$ implies that 
\begin{eqnarray*}
&&\sqrt{t_0}\sup_{(x,t)\in Q(x_0,\sqrt{t_0})}\int_{Q^c(x_0,\sqrt{t_0})}\left\arrowvert<\nabla^{g_0(t)}_xK_L(x,t,y,s),R_0(y,s)>\right \arrowvert d\mu_{g_0(s)}(y)ds\leq c\|R\|_{Y},\\
&&\sqrt{t_0}\sup_{(x,t)\in Q(x_0,\sqrt{t_0})}\int_{Q^c(x_0,\sqrt{t_0})}\left\arrowvert<\nabla^{g_0(s)}_y\nabla^{g_0(t)}_xK_L(x,t,y,s),R_1(y,s)>\right\arrowvert d\mu_{g_0(s)}(y)ds\leq c\|R\|_{Y}.
\end{eqnarray*}
That is why we assume from now on that the support of $R_0$ and $R_1$ is contained in a parabolic neighborhood of the form $Q_0:=Q(x_0,\sqrt{t_0})$.
\begin{itemize}
\renewcommand{\labelitemi}{$\bullet$}

\item 

[$R_0$ estimate]

Consider the following map : 
\begin{eqnarray*}
T_{K_L}(R_0)(x,t):=\int_{M\times \mathbb{R}_+}<\nabla^{g_0(t)}_xK_L(x,t,y,s)1_{s<t},R_0(y,s)>d\mu_{g_0(s)}ds.
\end{eqnarray*}
 In order to apply the Riesz convexity theorem, we show first that $T_{K_L}$ is of type $(1,p)$ and of type $(p',\infty)$ where $p'$ is the conjugate exponent of $p$, as soon as $p\in\left[1,\frac{n+2}{n+1}\right)$.

 $T_{K_L}$ is of type $(p',\infty)$ : indeed, by Hölder's inequality, one has
 \begin{eqnarray*}
\arrowvert T_{K_L}(R_0)\arrowvert(x,t)\leq \|\nabla^{g_0(t)}_xK_L(x,t,\cdot,\cdot)\|_{L^p(M\times (0,t))}\|R_0\|_{L^{p'}(Q_0)},
\end{eqnarray*}
for $(x,t)\in M\times \mathbb{R}_+^*.$
Now, by the Gaussian estimates established in theorem \ref{Gausian-est-scal-heat-ker} together with proposition \ref{cov-der-bound}, one gets by the co-area formula :
\begin{eqnarray*}
\|\nabla^{g_0(t)}_xK_L(x,t,\cdot,\cdot)1_{s<t}\|^p_{L^p(M\times(0,t))}&\leq& c\int_0^t\frac{(t-s)^{n/2}}{(t-s)^{\frac{(n+1)p}{2}}}ds\\
&\leq& ct^{\left(\frac{n}{2}+1\right)\left(1-p\right)+\frac{p}{2}},
\end{eqnarray*}
if $p\in\left[1,\frac{n+2}{n+1}\right)$ where $c$ is independent of time and space. Therefore, 
\begin{eqnarray*}
\|\sqrt{t}T_{K_L}(R_0)\|_{L^{\infty}(Q_0)}\leq c\nor tR_0\|_{L^{p'}(Q_0)}.\\
\end{eqnarray*}

$T_{K_L}$ is of type $(1,p)$ : indeed, by Minkowski's integral inequality, one has, if $p\in\left[1,\frac{n+2}{n+1}\right)$,
\begin{eqnarray*}
&&\|T_{K_L}(R_0)\|_{L^p(M\times(0,t_0))}\\
&\leq&\int_{M\times(0,t_0)}\left(\int_{M\times(0,t_0)}\|\nabla^{g_0(t)}_xK_L(x,t,y,s)1_{s<t}\|_g^p d\mu_t(x)dt\right)^{\frac{1}{p}}\arrowvert R_0\arrowvert_{g_0(s)}(y,s)d\mu_{g_0(s)}(y)ds\\
&\leq& ct_0^{\left(\frac{n}{2}+1\right)\left(\frac{1}{p}-1\right)+\frac{1}{2}}\|R_0\|_{L^1(Q_0)}, 
\end{eqnarray*}
which reads : 
\begin{eqnarray*}
\nor \sqrt{t_0}T_{K_L}(R_0)\|_{L^p(Q_0)}\leq c \nor t_0R_0\|_{L^1(Q_0)}.
\end{eqnarray*}
Therefore, by interpolation, $T_{K_L}$ is of type $(p_t,q_t)$ where $t\in[0,1]$ and $p_t, q_t$ are defined by 
\begin{eqnarray*}
\frac{1}{p_t}:=\frac{1-t}{1}+\frac{t}{p'},\quad \frac{1}{q_t}:=\frac{1-t}{p}+\frac{t}{\infty}.
\end{eqnarray*}
In particular, with $(p_t,q_t):=((n+4)/2, n+4)$ and $p:=\frac{n+4}{n+3}\in\left[1,\frac{n+2}{n+1}\right)$,
\begin{eqnarray*}
\nor \sqrt{t_0}T_K(R_0)\|_{L^{n+4}(Q_0)}\leq c \nor t_0R_0\|_{L^{\frac{n+4}{2}}(Q_0)}
\end{eqnarray*}

\item 

[$R_1$-estimate]
The proof is based on singular integrals defined on homogeneous spaces as in \cite{Koc-Lam-Rou}. Indeed, by theorem $2.4$ and section $3.2$ of \cite{Koc-Pde-Non-Euc}, the result follows if the following parabolic Riesz operator satisfies suitable estimates that are recalled below. 

Consider $M_+:=M\times \mathbb{R}_+$ endowed with the parabolic distance :
\begin{eqnarray*}
d(x,t,y,s):=\max\{d_{g_0(s)}(x,y),\sqrt{t-s}\},\quad x,y\in M,\quad 0\leq s\leq t.
\end{eqnarray*}
This turns $(M_+,d)$ into a homogeneous space, i.e. a complete metric space with a doubling measure.

Consider the operator $T_{K_L}(R_1):=\nabla^{g_0(\cdot)}(K_L\ast \nabla^{g_0(\cdot)}R_1)$, for $R_1$ with compact support in $Q_0$ as defined above.
This operator is an integral operator whose kernel $\mathcal{K}_L$ is given by the second covariant derivatives of $K_L$. By theorem \ref{Gausian-est-scal-heat-ker} together with proposition \ref{cov-der-bound}, $\mathcal{K}_L$ satisfies on $M_+$, 
\begin{eqnarray*}
&&\|\mathcal{K}_L(x,t,y,s)\|_g\leq cd(x,t,y,s)^{-(n+2)},\\
&&\min\{d(x,t,y,s)^{(n+2)},d(x',t',y',s')^{(n+2)}\|\mathcal{K}_L(x,t,y,s)-\mathcal{K}_L(x',t',y',s')\|_g\leq \\
&&c\frac{d(x,t,x',t')+d(y,s,y',s')}{d(x,t,y,s)+d(x',t',y',s')},
\end{eqnarray*}
for any $(x,t)$, $(y,s)$, $(x',t')$ and $(y',s')$ in $M_+$.
 
Moreover, by an integration by part on $M_+$, one can check that $T_{K_L}$ is continuous operator on $L^2(Q_0)$, i.e.
\begin{eqnarray*}
\|T_{K_L}(R_1)\|_{L^2(Q_0)}\leq c\|R_1\|_{L^2(Q_0)},
\end{eqnarray*}
where $c$ is a positive constant independent of $Q_0$  and $R_1$.
Therefore, by invoking theorem $2.4$ of \cite{Koc-Pde-Non-Euc}, $T_{K_L}$ is a continuous operator on $L^{n+4}(Q_0)$, which corresponds to the expected estimate after rescaling appropriately by a power of $t_0$.

\end{itemize}

\end{proof}
\end{proof}
To prove theorem \ref{main-theo}, it is now sufficient to invoke lemma \ref{lemma-contraction-map} together with theorem \ref{sec-est-inhom-equ} to show that the map $\Gamma$,
\begin{eqnarray*}
 h\in B_X(0_X,r)\stackrel{\Gamma}{\longmapsto} K_L\ast h(0)+K_L\ast R[h]\in B_X(0_X,r),
\end{eqnarray*}
is well-defined and is a contraction mapping for $r<1$ sufficiently small. Therefore, by the contraction mapping theorem, one gets a locally unique solution to the DeTurck Ricci flow with initial condition $g(0):=g_0+h(0)$.

\appendix

\section{Soliton equations}\label{sol-equ-sec}

The next lemma gathers well-known Ricci soliton identities together with the (static) evolution equations satisfied by the curvature tensor.

Recall first that an expanding gradient Ricci soliton is said \textit{normalized} if $\int_Me^{-f}d\mu_g=(4\pi)^{n/2}$ (whenever it makes sense).

\begin{lemma}\label{id-EGS}
Let $(M^n,g,\nabla^g f)$ be a normalized expanding gradient Ricci soliton. Then the trace and first order soliton identities are :
\begin{eqnarray}
&&\Delta_g f = \R_g+\frac{n}{2}, \label{equ:1} \\
&&\nabla^g \R_g+ 2\Ric(g)(\nabla^g f)=0, \label{equ:2} \\
&&\arrowvert \nabla^g f \arrowvert^2+\R_g=f+\mu(g), \label{equ:3}\\
&&\div_g\Rm(g)(Y,Z,T)=\Rm(g)(Y,Z, \nabla f,T),\label{equ:4}
\end{eqnarray}
for any vector fields $Y$, $Z$, $T$ and where $\mu(g)$ is a constant called the entropy.\\

The evolution equations for the curvature operator, the Ricci tensor and the scalar curvature are :
\begin{eqnarray}
&& \Delta_f \Rm(g)+\Rm(g)+\Rm(g)\ast\Rm(g)=0,\label{equ:5}\\
&&\Delta_f\Ric(g)+\Ric(g)+2\Rm(g)\ast\Ric(g)=0,\label{equ:6}\\
&&\Delta_f\R_g+\R_g+2\arrowvert\Ric(g)\arrowvert^2=0,\label{equ:7}
\end{eqnarray}
where, if $A$ and $B$ are two tensors, $A\ast B$ denotes any linear combination of contractions of the tensorial product of $A$ and $B$.
\end{lemma}

\begin{proof}
See [Chap.$1$,\cite{Cho-Lu-Ni-I}] for instance.

\end{proof}

\begin{prop}\label{pot-fct-est}
Let $(M^n,g,\nabla^g f)$ be an expanding gradient Ricci soliton.
\begin{itemize}
\item If $(M^n,g,\nabla^g f)$ is non Einstein, if $v:=f+\mu(g)+n/2$,
\begin{eqnarray}
&&\Delta_fv=v,\quad v>\arrowvert\nabla v\arrowvert^2.\label{inequ:1}\\
\end{eqnarray}

\item Assume $\Ric(g)\geq 0$ and assume $(M^n,g,\nabla^g f)$ is normalized. Then $M^n$ is diffeomorphic to $\mathbb{R}^n$ and
\begin{eqnarray}
&&v\geq \frac{n}{2}>0.\label{inequ:2}\\
&&\frac{1}{4}r_p(x)^2+\min_{M}v\leq v(x)\leq\left(\frac{1}{2}r_p(x)+\sqrt{\min_{M}v}\right)^2,\quad \forall x\in M,\label{inequ:3}\\
&&\AVR(g):=\lim_{r\rightarrow+\infty}\frac{\vol B(q,r)}{r^n}>0,\quad\forall q\in M,\label{inequ:avr}\\
&&-C(n,V_0,R_0)\leq\min_{M}f\leq 0\quad;\quad\mu(g)\geq\max_{M}\R_g\geq 0,\label{inequ:ent}
\end{eqnarray}
where $V_0$ is a positive number such that $\AVR(g)\geq V_0$, $R_0$ is such that $\sup_{M}\R_g\leq R_0$ and $p\in M$ is the unique critical point of $v$.\\

\item Assume $\Ric(g)=\textit{O}(r_p^{-2})$ where $r_p$ denotes the distance function to a fixed point $p\in M$. Then the potential function is equivalent to $r_p^2/4$ (up to order $2$).
\end{itemize}
\end{prop}

For a proof, see \cite{Der-Asy-Com-Egs} and the references therein.

\begin{lemma}[Distance distortions]\label{C^0-distance-est-Type-III}
Let $(M^n,g_0(t))_{t\geq 0}$ be a Type III solution of the Ricci flow with nonnegative Ricci curvature, i.e.
\begin{eqnarray*}
\arrowvert\Rm(g_0(t))\arrowvert_{g_0(t)}\leq \frac{R_0}{1+t},\quad\Ric(g_0(t))\geq 0,\quad t\geq 0.
\end{eqnarray*}
Then, 
\begin{eqnarray*}
\left(\frac{1+s}{1+t}\right)^{c(R_0)}g_0(s)&\leq& g_0(t)\leq g_0(s),\\
d_{g_0(s)}(x,y)-c(R_0)\left(\sqrt{t}-\sqrt{s}\right)\leq d_{g_0(t)}(x,y)&\leq& d_{g_0(s)}(x,y),\quad s\leq t,\quad x,y\in M,
\end{eqnarray*}
for a positive constant $c(R_0)$, and any $0\leq s\leq t$.

In particular, 
\begin{eqnarray*}
-\frac{c(R_0)}{\sqrt{t}}\leq\partial_td_{g_0(t)}(x,y)\leq 0, \quad t>0,\quad x,y \in M,
\end{eqnarray*}
for a positive constant $c(R_0)$.

\end{lemma}
\begin{proof}
The upper bound comes from the fact that the Ricci curvature is nonnegative, hence, $g_0(t)\leq g_0(s)$ for any $s\leq t$ in the sense of symmetric $2$-tensors. 

The lower bound has been proved by Hamilton and a proof can be found for instance in [lemma $8.33$, \cite{Ben}].
\end{proof}

\bibliographystyle{alpha.bst}
\bibliography{bib-wea-sta-egs}

\end{document}